\documentclass[letterpaper,11pt,reqno]{amsart}
\usepackage[portrait,margin=1in]{geometry}
\usepackage[T1]{fontenc}
\usepackage{mathrsfs,xfrac} 
\usepackage[colorlinks=true,linkcolor=blue,citecolor=blue,urlcolor=blue]{hyperref} 
\usepackage{amsmath,amssymb,amsthm,amsfonts,amsbsy,latexsym,dsfont,color,graphicx,enumitem,microtype}
\usepackage[foot]{amsaddr}
\usepackage{caption}
\usepackage{regexpatch}
\usepackage{todonotes}
\usepackage{comment}
\makeatletter
\xpatchcmd{\@todo}{\setkeys{todonotes}{#1}}{\setkeys{todonotes}{inline,#1}}{}{}
\makeatother

\newenvironment{enumeratei}{\begin{enumerate}[label=\textup{\roman*.},leftmargin=2em,itemsep=1ex]}{\end{enumerate}}
\newenvironment{enumeratea}{\begin{enumerate}[label=\textup{\alph*)},leftmargin=2em,itemsep=1ex]}{\end{enumerate}}

\newtheorem{thm}{Theorem}[section]
\newtheorem{lem}[thm]{Lemma}

\newtheorem{prop}[thm]{Proposition}

\newtheorem{rem}[thm]{Remark}

\renewcommand{\le}{\leqslant}  
\renewcommand{\ge}{\geqslant} 
\renewcommand{\leq}{\leqslant}  
\renewcommand{\geq}{\geqslant}

\newcommand{\wt}{\widetilde}

\newcommand{\ind}{\mathds{1}}
\newcommand{\eps}{\varepsilon}

\newcommand{\norm}[1]{\left\Vert#1\right\Vert}
\newcommand{\abs}[1]{\left\vert#1\right\vert}

\renewcommand{\iff}{\Leftrightarrow}

    \let\gk=\kappa

\newcommand{\cF}{\mathcal{F}}
\newcommand{\cG}{\mathcal{G}}\newcommand{\cI}{\mathcal{I}}
\newcommand{\cJ}{\mathcal{J}}\newcommand{\cK}{\mathcal{K}}\newcommand{\cL}{\mathcal{L}}
\newcommand{\cO}{\mathcal{O}}
\newcommand{\cP}{\mathcal{P}}
\newcommand{\cS}{\mathcal{S}}\newcommand{\cT}{\mathcal{T}}
\newcommand{\cV}{\mathcal{V}}\newcommand{\cX}{\mathcal{X}}
\newcommand{\cZ}{\mathcal{Z}}  
\newcommand{\vone}{\mathbf{1}}

\newcommand{\mvg}{\boldsymbol{g}}

\newcommand{\mvn}{\boldsymbol{n}}
\newcommand{\mvq}{\boldsymbol{q}}

\newcommand{\mvgs}{\boldsymbol{\sigma}}\newcommand{\mvgt}{\boldsymbol{\tau}}

\newcommand{\dR}{\mathds{R}}

\newcommand{\dZ}{\mathds{Z}} 
\newcommand{\sR}{\mathscr{R}}

\DeclareMathOperator{\E}{\mathds{E}}
\DeclareMathOperator{\pr}{\mathds{P}}
\DeclareMathOperator{\N}{N}

\DeclareMathOperator{\var}{Var}

\DeclareMathOperator{\tr}{Tr}

\newcommand{\fR}{\mathfrak{R}}
\DeclareMathOperator{\bal}{bal}
\allowdisplaybreaks[2]

\begin{document}

\title[Potts Spin Glass Model]{Fluctuations in the Potts spin glass model at high temperature}
\author[Dey]{Partha S.~Dey$^{\star}$}
\author[Kang]{Taegu Kang$^{\dagger}$}
\address{$^{\star,\dagger}$Department of Mathematics, University of Illinois Urbana-Champaign}
\email{\{$^\star$psdey,$^\dagger$taeguk2\}@illinois.edu}
\date{}

\begin{abstract}
	We study the fluctuations of the free energy of the Potts spin glass with $\gk$ colors at high temperature. The disorder is symmetric with unit variance and finite fourth moment. For configurations with a balanced, or nearly balanced, color profile, the log-partition function satisfies a central limit theorem whenever $\beta^2<\alpha_c(\gk)$. Here $\gk\alpha_c(\gk)$ is the critical inverse temperature of the Curie--Weiss--Potts model. A comparison of color profiles transfers this limit to the unconstrained model for $\beta<\beta_{\max}(\gk)$, an explicit threshold. In the same range, the magnetizations of two replicas and their overlap matrix satisfy a joint quenched central limit theorem with independent Gaussian limits. The fourth cumulant of the disorder enters the limiting mean and variance explicitly.
\end{abstract}

\maketitle
%\setcounter{tocdepth}{1}{\setlength{\parskip}{0pt}\tableofcontents}

%%%%%%%%%%%%%%%%%%%%%%%%%%%%%%%%%%%%%%%%%%%%%%%%%%%%%%%%%%%%%%%%%%%%%%%%%%%%%%%
\section{Introduction}
%%%%%%%%%%%%%%%%%%%%%%%%%%%%%%%%%%%%%%%%%%%%%%%%%%%%%%%%%%%%%%%%%%%%%%%%%%%%%%%

\subsection{Model and motivation}\label{sec:model}
The Potts spin glass is the $\gk$-state generalization of the Sherrington--Kirkpatrick (SK) model. It was introduced by Elderfield and Sherrington~\cite{ES83} to describe disordered magnets whose local state takes $\gk$ values, and subsequently studied by Gross, Kanter and Sompolinsky~\cite{GKS85}. It is also a mean-field model for the structural glass transition~\cite{KW87}. For $\gk\geq3$, its phase diagram differs qualitatively from that of the SK model. For instance, the replica analysis predicts a discontinuous spin-glass transition for $\gk>4$ in the color-symmetric setting~\cite{GKS85}. The model also arises in combinatorial optimization and statistical inference. The ground-state energy of the balanced Gaussian model gives the leading correction to the maximum $\gk$-cut of Erd\H{o}s--R\'enyi and random regular graphs in the large-degree limit~\cite{Sen18}; see~\cite{DMS17} for the SK case. Generalized Potts models also describe posterior distributions of community labels in stochastic block models~\cite{DKMZ11}. We refer to~\cite{Wu82} for background on the Potts model.

Fix the number of colors $\gk\ge 2$, and let $\Sigma_N := [\gk]^N$ be the space of color configurations $\mvgs = (\sigma_i)_{i \in [N]}$. The Hamiltonian of the Potts spin glass is
\begin{align}
	H_N (\mvgs) := \frac{\beta}{\sqrt{N}} \sum_{i < j } J_{ij}  A_{ij}(\mvgs),
	\quad\text{where }
	A_{ij}(\mvgs) := \gk \cdot \ind \{ \sigma_i = \sigma_j\} - 1.
	\label{eq:hamiltonian}
\end{align}
The couplings $(J_{ij})_{i<j}$ are i.i.d.~copies of a symmetric random variable $J$ with $\E(J^2)=1$ and $\E(J^4)<\infty$. The \emph{edge weight} $A_{ij}(\mvgs) \in \{-1, \gk-1\}$ records whether the endpoints of the edge $\{i,j\}$ carry the same color. It is centered under the uniform measure on $\Sigma_N$. The partition function and the Gibbs measure at inverse temperature $\beta>0$ are
\begin{align}
	Z_N(\beta):=\sum_{\mvgs\in\Sigma_N}e^{H_N(\mvgs)}
	\quad\text{and}\quad
	\varrho_{N,\beta}(\mvgs):=\frac{e^{H_N(\mvgs)}}{Z_N(\beta)}.\nonumber
\end{align}
We write $Z_N(\beta;J)$ when the dependence on the disorder matters. Gibbs averages are denoted by $\langle\cdot\rangle_{\varrho_{N,\beta}}$, and $\varrho_{N,\beta}^{\otimes2}$ is the law of two independent replicas. The dependence on the fixed $\beta$ is often suppressed.

Three features of~\eqref{eq:hamiltonian} deserve comment.
\begin{enumeratea}
	\item \emph{Vector spins.} With the field $\phi$ defined in~\eqref{def:phi} below, $A_{ij}(\mvgs)=\phi(\sigma_i)^{\intercal}\phi(\sigma_j)$ by Lemma~\ref{lem:edge-id}. Thus $H_N$ is the SK Hamiltonian for the vector spins $\phi(\sigma_i)$, which point to the vertices of a regular simplex. This is the vector-spin point of view of~\cite{Pan18}. For $\gk=2$, identifying the two colors with $\pm1$ gives $A_{ij}(\mvgs)=\sigma_i\sigma_j$, and $H_N$ is the SK Hamiltonian.
	\item \emph{Centering.} Replacing $A_{ij}$ by $\gk\ind\{\sigma_i=\sigma_j\}$ does not change the Gibbs measure. It adds the configuration-independent term $\beta N^{-1/2}\sum_{i<j}J_{ij}$ to $\ln Z_N(\beta)$, and this term has variance $\beta^2(N-1)/2$. The centering removes this trivial source of fluctuations, which is essential for order-one fluctuation results. The same centering, $\ind\{\sigma_i=\sigma_j\}-\gk^{-1}=A_{ij}(\mvgs)/\gk$, is used in~\cite{Kim26}, where it is also shown to be needed for the second-moment method.
	\item \emph{General disorder.} Under mild moment assumptions, the limiting free energy of the SK model does not depend on the law of the couplings~\cite{CH06}, but its fluctuations do. In our results the fourth cumulant $\mu_4:=\E(J^4)-3$ of the disorder enters the limiting mean and variance.
\end{enumeratea}

In this article we study the fluctuations of the log-partition function $\ln Z_N(\beta)$ in the high temperature regime. A central role is played by the color profile. Write $n(\mvgs)=\bigl(\abs{\{i:\sigma_i=x\}}\bigr)_{x\in[\gk]}$ for the color count vector. For a profile $\mvn\in\dZ_{\geq0}^{\gk}$ with $\sum_x n_x=N$, we define
\begin{align}\label{def:profile}
	\Sigma_N^{(\mvn)}:=\{\mvgs\in\Sigma_N:n(\mvgs)=\mvn\}
	\quad\text{and}\quad
	Z_N^{(\mvn)}(\beta):=\sum_{\mvgs\in\Sigma_N^{(\mvn)}}e^{H_N(\mvgs)}.
\end{align}

Let $u:=\gk^{-1}\vone$ be the uniform color law, $e_x$ the $x$-th coordinate vector, and $I$ the identity matrix. Define
\begin{align}\label{def:phi}
	\phi(x):=\sqrt{\gk}\,(e_x-u),\quad
	\Pi:=I-\gk^{-1}\vone\vone^{\intercal},
	\quad\text{and}\quad
	\cV:=\Big\{v\in\dR^\gk:\textstyle\sum_xv_x=0\Big\}.
\end{align}
The field $\phi$ takes values in $\cV$, and $\Pi$ is the orthogonal projection onto $\cV$. The operator $\Pi\otimes\Pi$ acts on $\gk\times\gk$ matrices by $s\mapsto\Pi s\Pi$. It is the orthogonal projection onto $\cV\otimes\cV$ and has rank $(\gk-1)^2$.
We write the color profile and the centered magnetization as
\begin{align}
	p(\mvgs):=\frac{n(\mvgs)}N
	\quad\text{and}\quad
    m(\mvgs):=\frac1{\sqrt N}\sum_{i=1}^N\phi(\sigma_i)=\sqrt{\gk N}\,\big(p(\mvgs)-u\big).
	\label{eq:profile-mag}
\end{align}
The magnetization is constant on $\Sigma_N^{(\mvn)}$, with value
$$
    m_{\mvn}:=\sqrt{\gk/N}\,\big(\mvn-(N/\gk)\vone\big).
$$
We fix a most balanced profile $\mvn^*=\mvn^*(N)$, whose counts are all equal to $\lfloor N/\gk\rfloor$ or $\lceil N/\gk\rceil$. Then $\norm{m_{\mvn^*}}_2\le \gk N^{-1/2}$. When $\gk\mid N$, we write
\begin{align*}
    \Sigma_N^{\bal}:=\Sigma_N^{(\mvn^*)}
    \quad\text{and}\quad
    Z_N^{\bal}:=Z_N^{(\mvn^*)},
\end{align*}
and $\varrho_N^{\bal}$ for the Gibbs measure conditioned on $\Sigma_N^{\bal}$.
For two configurations $\mvgs,\mvgt\in\Sigma_N$, the pair empirical measure $p(\mvgs,\mvgt)$ has entries
\[
p_{xy}(\mvgs,\mvgt):=N^{-1}\abs{\{i:\sigma_i=x,\ \tau_i=y\}}, \quad x,y\in[\gk].
\]
The number of arguments distinguishes it from the color profile $p(\mvgs)$. The \emph{scaled overlap matrix} is
\begin{align}
	\sR(\mvgs,\mvgt):=\frac1{\sqrt N}\sum_{i=1}^N\phi(\sigma_i)\phi(\tau_i)^{\intercal}
		=\gk\sqrt N\,\Pi p(\mvgs,\mvgt)\Pi.
	\label{eq:pair-profile}
\end{align}

The second equality holds because $\sum_{x,y}q_{xy}\phi(x)\phi(y)^{\intercal}=\gk\,\Pi q\Pi$ for every $\gk\times\gk$ matrix $q$. Both $m$ and $\sR$ are centered fluctuation fields on the $\sqrt N$ scale. 
Finally, let $E_N$ be the edge set of the complete graph on $[N]$. For $e=\{i,j\}\in E_N$ we write $A_e(\mvgs):=A_{ij}(\mvgs)$ and $J_e:=J_{ij}$. A \emph{loop} of length $n\geq3$ is a simple cycle with $n$ edges in this graph, regarded as a set of edges.

\subsection{Main results}\label{sec:main-results}
Define
\begin{align}\label{def:alpha}
    \alpha:=\frac{(\gk-2)\beta^2}{2}.
\end{align}
The parameter $\alpha$ measures how strongly the disorder couples to the color profile. By Lemma~\ref{lem:general-kernels} below, the disorder average of $e^{H_N(\mvgs)}$ (for the truncated disorder) is proportional to $e^{\alpha\norm{m(\mvgs)}_2^2/2}$, up to a factor $1+o(1)$, for configurations with bounded magnetization. In particular $\alpha=0$ for $\gk=2$.

The thresholds are expressed through the optimal sub-Gaussian constant $\eta_\phi(\gk)$ of $\phi(\sigma)$ with $\sigma$ uniform on $[\gk]$. By Lemma~\ref{lem:eta-phi},
\begin{align}\label{def:eta-phi}
	\eta_\phi(2)=1
	\quad\text{and}\quad
	\eta_\phi(\gk)=\frac{\gk(\gk-2)}{2(\gk-1)\ln(\gk-1)}\quad\text{for }\gk\geq3.
\end{align}
We define
\begin{align}\label{def:betamax}
	\alpha_c(\gk):=\eta_\phi(\gk)^{-1}
	\quad\text{and}\quad
	\beta_{\max}(\gk):=\sqrt{\alpha_c(\gk)}\,
	\min\left(1,\sqrt{\frac2{\gk-2}}\right),
\end{align}
where the second entry of the minimum is $+\infty$ for $\gk=2$.
For $\gk\geq3$, the number $\gk\alpha_c(\gk)=2(\gk-1)\ln(\gk-1)/(\gk-2)$ is the critical inverse temperature of the Curie--Weiss--Potts model~\cite{EW90}. We have $\alpha_c(2)=1$ and $\alpha_c(\gk)<1$ for $\gk\geq3$. The condition $\beta<\beta_{\max}(\gk)$ is equivalent to the two conditions $\beta^2<\alpha_c(\gk)$ and $\alpha<\alpha_c(\gk)$. The first controls the overlap and the second controls the color profile, see Remark~\ref{rem:temperature-range}.

For $0\leq x<1$ we define
\begin{align}
	\varphi(x)&:=-\tfrac12\big(x+\ln(1-x)\big)=\sum_{n\geq2}\frac{x^n}{2n},
	\quad
	\varphi_\gk(x):=(\gk-1)^2\varphi(x),
	\quad\text{and} \label{eq:varphi}\\
	\varphi_{\gk,J}(x) &:=\varphi_\gk(x)+\frac{\mu_4(\gk-1)^2x^2}{8}
	=\frac{(\gk-1)^2x^2}{8}\big(\E(J^4)-1\big)
		+(\gk-1)^2\sum_{n\geq3}\frac{x^n}{2n}.
	\label{eq:general-variance}
\end{align}
The last expression shows that $\varphi_{\gk,J}\geq0$, since $\E(J^4)\geq\E(J^2)^2=1$. Our first result is a central limit theorem, with fluctuations of order one, for the log-partition function restricted to a near-balanced profile.

\begin{thm}[Balanced and near-balanced free energy]\label{thm:balanced_clt}
	Let $0<\beta^2<\alpha_c(\gk)$ and let $(\mvn_N)$ be a sequence of profiles with $\sup_N\norm{m_{\mvn_N}}_2<\infty$. Then, as $N\to\infty$,
	\begin{align*}
		&\ln Z_N^{(\mvn_N)}(\beta)-\ln\abs{\Sigma_N^{(\mvn_N)}}
			-\frac{\beta^2(\gk-1)}4(N-\gk+1)-\frac\alpha2\norm{m_{\mvn_N}}_2^2\\
		&\quad\Longrightarrow
		\N\left(-\frac12\varphi_\gk(\beta^2)
			+\frac{\mu_4\beta^4(\gk-1)(\gk^2-6\gk+6)}{48},
			\ \varphi_{\gk,J}(\beta^2)\right).
	\end{align*}
\end{thm}

Combining Theorem~\ref{thm:balanced_clt} with the concentration of the color profile, we obtain the same type of limit for the unconstrained model.
\begin{thm}[Unconstrained free energy]\label{thm:full_clt}
	Let $0<\beta<\beta_{\max}(\gk)$. Then, as $N\to\infty$,
	\begin{align*}
		&\ln Z_N(\beta)-N\ln\gk-\frac{\beta^2(\gk-1)}4(N-1)\\
		&\quad\Longrightarrow
		\N\left((\gk-1)\varphi(\alpha)-\frac12\varphi_\gk(\beta^2)
			+\frac{\mu_4\beta^4(\gk-1)(\gk^2-6\gk+6)}{48},
			\ \varphi_{\gk,J}(\beta^2)\right).
	\end{align*}
\end{thm}

Theorem~\ref{thm:full_clt} follows from Theorem~\ref{thm:balanced_clt} by a comparison of profiles, and the two limiting variances coincide. After accounting for the different centerings, the limiting mean in Theorem~\ref{thm:full_clt} exceeds that in Theorem~\ref{thm:balanced_clt} by $(\gk-1)\varphi(\alpha)$. This deterministic excess is the contribution of the Gaussian fluctuations of the color profile, see~\eqref{eq:general-profile-tilt}.
Our last result describes the joint behavior of the magnetizations of two replicas and of their overlap matrix.

\begin{thm}[Joint quenched central limit theorem]\label{thm:C}
	Let $0<\beta<\beta_{\max}(\gk)$, and let
    \[
    g,g'\sim\N\big(0,\Pi/(1-\alpha)\big)\quad\text{and}\quad K\sim\N\big(0,\Pi\otimes\Pi/(1-\beta^2)\big)
    \]
    be independent. Then, for every bounded continuous $F$ on $\dR^\gk\times\dR^\gk\times\dR^{\gk\times\gk}$,
	\begin{align*}
		\left\langle F\big(m(\mvgs^1),m(\mvgs^2),\sR(\mvgs^1,\mvgs^2)\big)
			\right\rangle_{\varrho_{N,\beta}^{\otimes2}}
		\longrightarrow \E\big(F(g,g',K)\big)
		\quad\text{in probability}.
	\end{align*}
\end{thm}

\begin{rem}[Special cases]\label{rem:special-cases}
	\begin{enumeratea}
	    \item Theorem~\ref{thm:balanced_clt} applies to $Z_N^{\bal}$ along $\gk\mid N$, where $m_{\mvn^*}=0$. It also applies to $Z_N^{(\mvn^*)}$ for every $N$, where the magnetization correction vanishes in the limit.
        \item For Gaussian disorder $\mu_4=0$, so $\varphi_{\gk,J}=\varphi_\gk$ and the fourth-cumulant term in the mean vanishes.
        \item For $\gk=2$ we have $\alpha=0$, $\alpha_c(2)=\beta_{\max}(2)=1$, and $\varphi_2=\varphi$. For Gaussian disorder, Theorem~\ref{thm:full_clt} then states that $\ln Z_N(\beta)-N\ln2-\beta^2(N-1)/4\Longrightarrow\N\big(-\tfrac12\varphi(\beta^2),\varphi(\beta^2)\big)$ for $\beta<1$. This is the central limit theorem of Aizenman, Lebowitz and Ruelle~\cite{ALR87} for the SK model. For Rademacher couplings $\E(J^4)=1$, and the limiting variance reduces to the loop contribution $\sum_{n\geq3}\beta^{2n}/(2n)$.
	\end{enumeratea}
\end{rem}

\begin{rem}[The two temperature conditions]\label{rem:temperature-range}
	The two conditions defining $\beta_{\max}(\gk)$ enter at different stages of the proof. a)~\textbf{The overlap condition $\beta^2<\alpha_c(\gk)$}. This condition controls exponential moments of the overlap between uniform configurations with bounded magnetizations (Lemma~\ref{lem:profile-moments}). These bounds yield the mixed second-moment limits for the normalized partition functions and the overlap comparison. This is the only temperature restriction needed for Theorem~\ref{thm:balanced_clt}.
    b)~\textbf{The profile condition $\alpha<\alpha_c(\gk)$}. This additional condition controls the contribution of profiles with large magnetization when removing the profile constraint. For the truncated disorder, the annealed profile weights on bounded magnetization windows are asymptotically proportional to $P_{\mvn}e^{\alpha\norm{m_{\mvn}}_2^2/2}$, where $P_{\mvn}$ is the multinomial probability of $\mvn$. These leading weights are those of the Curie--Weiss--Potts model at inverse temperature $\gk\alpha$, and $\alpha<\alpha_c(\gk)$ places that model in its high-temperature phase~\cite{EW90}. The resulting tail bounds allow the comparison to extend to all profiles (Proposition~\ref{prop:profile-equiv}).
\end{rem}

\subsection{Ideas behind the proofs}\label{sec:ideas}

Near-balanced profiles asymptotically share the same normalized random factor. We identify its law using quadratic and loop statistics, then sum the deterministic profile weights to obtain the unconstrained free energy. An overlap comparison gives the joint quenched limit.

We first replace $J$ by the truncated disorder $\wt J$ of Section~\ref{sec:general-kernels}, so that the exponential moments needed for normalization and tilting exist. The two disorder arrays agree with probability $1-o(1)$ by~\eqref{eq:truncation-coupling}, allowing all conclusions to transfer back to $J$.

\smallskip\noindent\textbf{Step 1: a common random factor.} For each profile we normalize by the disorder average,
\begin{align*}
 \cZ_N^{(\mvn)}:=\frac{Z_N^{(\mvn)}(\beta;\wt J)}{\E\big(Z_N^{(\mvn)}(\beta;\wt J)\big)}.
\end{align*}
After averaging over the disorder, the mixed second moments are expressed through overlaps of independent uniform configurations with prescribed profiles. The overlap CLT and exponential-moment bounds give the same limit for all profile pairs in a bounded magnetization window (Lemma~\ref{lem:general-profile-moments}). Expanding the square therefore gives, for every fixed $L>0$,
\begin{align*}
 \sup_{\norm{m_{\mvn}}_2\leq L}
 \E\Big(\big(\cZ_N^{(\mvn)}-\cZ_N^{(\mvn^*)}\big)^2\Big)\longrightarrow0.
\end{align*}

\smallskip\noindent\textbf{Step 2: quadratic and loop statistics (Theorem~\ref{thm:balanced_clt}).} We approximate the common factor by $e^{\wt G_{N,\ell}-v_\ell/2}$, where $\wt G_{N,\ell}$ is the weighted sum~\eqref{eq:general-loop-statistic} of a centered quadratic statistic and loop statistics of lengths $3,\ldots,\ell$. The quadratic term comes from squared couplings, and the loop terms from products of couplings around cycles; Remark~\ref{rem:loop-motivation} explains this choice. Write $v_\ell:=\varphi_{\gk,J,\ell}(\beta^2)$ for its limiting variance, with $v_\ell\uparrow v:=\varphi_{\gk,J}(\beta^2)$.

The mixed term in the squared error can be evaluated by tilting the disorder with the normalized Boltzmann weight of a fixed configuration. We prove
\begin{align*}
 \wt G_{N,\ell}&\Longrightarrow\N(0,v_\ell)
 &&\text{under the product law, and}\\
 \wt G_{N,\ell}&\Longrightarrow\N(v_\ell,v_\ell)
 &&\text{under the tilted law of a near-balanced configuration}.
\end{align*}
With the bounded cutoff $f_{\ell,B}(x):=\min\{e^{x-v_\ell/2},B\}$, these limits and the second-moment limit $\E\big((\cZ_N^{(\mvn^*)})^2\big)\to e^v$ give
\begin{align*}
 \lim_{B\to\infty}\lim_{N\to\infty}
 \E\Big(\big(\cZ_N^{(\mvn^*)}-f_{\ell,B}(\wt G_{N,\ell})\big)^2\Big)
 =e^v-e^{v_\ell}\longrightarrow0\quad(\ell\to\infty).
\end{align*}
Thus the quadratic and loop statistics account for the limiting fluctuations. Consequently $\ln\cZ_N^{(\mvn^*)}\Longrightarrow\N(-v/2,v)$, and the deterministic normalization gives Theorem~\ref{thm:balanced_clt}.

\smallskip\noindent\textbf{Step 3: removing the profile constraint (Theorem~\ref{thm:full_clt}).} On bounded magnetization windows, Step 1 lets us replace the normalized profile partition functions by the common factor. Under $\alpha<\alpha_c(\gk)$, exponential-moment bounds control the remaining profiles and extend this replacement to the full profile sum. The common factor cancels asymptotically from the ratio, which is determined by the annealed profile weights, whose leading dependence on magnetization is $e^{\alpha\norm{m_{\mvn}}_2^2/2}$. The magnetization CLT under uniform colors and the same tail bounds give
\begin{align*}
 \frac{Z_N(\beta;\wt J)}{Z_N^{(\mvn^*)}(\beta;\wt J)}\cdot
 \frac{\abs{\Sigma_N^{(\mvn^*)}}}{\gk^N}
 \longrightarrow(1-\alpha)^{-(\gk-1)/2}
 \quad\text{in probability}.
\end{align*}
Taking logarithms gives the deterministic difference between the two free energies.

\smallskip\noindent\textbf{Step 4: quenched overlaps (Theorem~\ref{thm:C}).} For fixed near-balanced profiles, the disorder-averaged pair weight produces a quadratic tilt of the uniform Gaussian overlap limit, giving covariance $\Pi\otimes\Pi/(1-\beta^2)$. For quenched convergence we follow Comets and Neveu~\cite{CN95}. For test functions centered under the uniform Gaussian law, Cauchy--Schwarz controls a reweighted, unnormalized pair average through three configurations. The asymptotic independence of their pairwise overlaps makes the centered error vanish. After normalization, this gives the quenched overlap limit. The profile comparison and tail bounds from Step 3 then allow us to average over profiles, yielding, for bounded continuous $h_1,h_2,f$,
\begin{align*}
 \Big\langle h_1(m(\mvgs^1))h_2(m(\mvgs^2))f(\sR(\mvgs^1,\mvgs^2))\Big\rangle_{\varrho_{N,\beta}^{\otimes2}}
 \longrightarrow \E\big(h_1(g)\big)\E\big(h_2(g')\big)\E\big(f(K)\big)
 \quad\text{in probability},
\end{align*}
with $g,g',K$ as in Theorem~\ref{thm:C}. This factorization gives the joint limit and independence.

\begin{rem}[Motivation for the quadratic and loop approximation]\label{rem:loop-motivation}
 The graph expansion of the unconstrained partition function motivates the disorder statistics used for the common factor in Step 1. Define
 \begin{align*}
  \mathrm{Ch}_\gk(x)&:=e^{-x}\big(1+\gk^{-1}(e^{\gk x}-1)\big)
  \quad\text{and}\quad
  \mathrm{Th}_\gk(x):=\frac{e^{\gk x}-1}{e^{\gk x}+\gk-1}.
 \end{align*}
 The identity $e^{xA_e(\mvgs)}=\mathrm{Ch}_\gk(x)\big(1+\mathrm{Th}_\gk(x)A_e(\mvgs)\big)$ gives the exact factorization
 \begin{align*}
  Z_N(\beta;J)&=\gk^N Z_N^{\mathrm{ch}}(\beta)Z_N^{\mathrm{cl}}(\beta),
  \quad\text{where}\\
  Z_N^{\mathrm{ch}}(\beta):=\prod_{e\in E_N}\mathrm{Ch}_\gk(\beta J_e/\sqrt N),\quad &
  Z_N^{\mathrm{cl}}(\beta):=\E_{\mvgs}\Big(\prod_{e\in E_N}
  \big(1+ \omega_eA_e(\mvgs)\big)\Big),
 \end{align*}
 where $\E_{\mvgs}$ is the uniform average on $\Sigma_N$ and $\omega_e:=\mathrm{Th}_\gk(\beta J_e/\sqrt N)$. For $\Gamma \subseteq E_N$ write $\omega(\Gamma):=\prod_{e\in\Gamma}\omega_e$. Expanding the product yields
 \begin{align*}
  Z_N^{\mathrm{cl}}(\beta)=\sum_{\Gamma\subseteq E_N}f_\gk(\Gamma)\omega(\Gamma),
  \quad f_\gk(\Gamma):=\E_{\mvgs}\Big(\prod_{e\in\Gamma}A_e(\mvgs)\Big).
 \end{align*}
 If $\Gamma$ has a vertex of degree one, averaging over its color gives $f_\gk(\Gamma)=0$. For a loop $\gamma$ of length $n$, $f_\gk(\gamma)=\tr(\Pi^n)=\gk-1$.

 To identify the leading random terms, use $\mathrm{Th}_\gk(x)=x+\cO(x^2)$ and replace $\omega_e$ by the centered weight $\beta J_e/\sqrt N$. For a fixed graph with $v$ nonisolated vertices and $r$ edges, the sum over its copies has variance $\cO(N^{v-r})$, since distinct edge sets give orthogonal monomials. The contributing nonempty graphs have minimum degree at least two, so $v\leq r$, with equality precisely for vertex-disjoint unions of loops. Their coefficients and edge products factor over the loops, suggesting the formal approximation
 \begin{align*}
  Z_N^{\mathrm{cl}}(\beta)\approx Z_{N,\ell}^{\mathrm{loop}}(\beta)
  :=\prod_{\substack{\gamma\text{ loop}\\3\leq\abs\gamma\leq\ell}}
  \big(1+(\gk-1)\omega(\gamma)\big).
 \end{align*}
 The linear term in the formal expansion of $\ln Z_{N,\ell}^{\mathrm{loop}}$, with each $\omega_e$ replaced by $\beta J_e/\sqrt N$, gives $(\gk-1)\sum_{n=3}^{\ell}\beta^nY_{N,n}(J)$. The expansion $\ln\mathrm{Ch}_\gk(x)=\frac{\gk-1}{2}x^2+\cO(x^3)$ supplies the centered quadratic term $(\gk-1)\beta^2X_N(J)/2$, with $X_N$ and $Y_{N,n}$ defined in~\eqref{eq:array-cycle-statistics}. After truncation, these terms give the statistic used in the proof,
 \begin{align*}
  \wt G_{N,\ell}
  =\frac{(\gk-1)\beta^2}{2}\wt X_N
  +(\gk-1)\sum_{n=3}^{\ell}\beta^n\wt Y_{N,n},
 \end{align*}
 where $\wt X_N$ and $\wt Y_{N,n}$ are defined in~\eqref{eq:truncated-loop-statistics}. At $\gk=2$, these are the quadratic and loop terms of the cluster expansion in~\cite{ALR87}.
\end{rem}

\subsection{Related work}
Panchenko~\cite{Pan18} proved a Parisi-type variational formula for the free energy at all temperatures. For the balanced model, Bates and Sohn~\cite{BS24} reduced the order parameter to that of the SK model, proving a prediction of Elderfield and Sherrington~\cite{ES83}. For the model with a self-overlap correction, Chen~\cite{Che24} obtained a Parisi formula over scalar-valued paths and proved uniqueness of the Parisi measure. Mourrat~\cite{Mou25} proved that color symmetry breaks in the unconstrained model on a nonempty interval of inverse temperatures for $\gk\geq58$. Chen~\cite{Che25} related color symmetry breaking to a ferromagnetic transition when a ferromagnetic interaction is added. Kim~\cite{Kim26} recently showed, for $\gk\geq3$, that color symmetry is preserved at high temperature and that the limiting free energy equals the annealed value of the centered model. The proof uses a second-moment method on the balanced model with a centered Hamiltonian. The centered Hamiltonian used in~\cite{Kim26} sums Gaussian couplings over ordered pairs of sites with the weights $\ind\{\sigma_i=\sigma_j\}-\gk^{-1}$. Matching its off-diagonal interactions with ours amounts to replacing its inverse temperature by $\gk\beta/\sqrt2$. With this substitution, the threshold~\cite[(1.3)]{Kim26} becomes $\beta_{\max}(\gk)$.

For the SK model at high temperature and without external field, $\ln Z_N$ has Gaussian fluctuations of order one. The two classical proofs are the cluster expansion of Aizenman, Lebowitz and Ruelle~\cite{ALR87} and the stochastic calculus approach of Comets and Neveu~\cite{CN95}, see also~\cite{Tal11}. Our proof of Theorem~\ref{thm:balanced_clt} is closer in spirit to the first, and our proof of Theorem~\ref{thm:C} adapts the second. The literature on free-energy fluctuations in this family of models is by now large. It covers other temperature regimes, external fields, and spherical and multi-species variants such as~\cite{DW21}. We refer to the survey of Collins-Woodfin and Le~\cite{CL25}. Finally, our second-moment argument with loop statistics is a Gaussian analogue of the small subgraph conditioning method of Robinson and Wormald~\cite{RW94} and Janson~\cite{Jan95}. Banerjee~\cite{Ban20} used a Gaussian version of this method to prove free-energy fluctuations for the SK model with Curie--Weiss interaction, combining the Gaussian limits of cycle statistics under the original and tilted disorder laws with a second-moment bound, see also~\cite{DW21,DK26b}. Similar loop (cycle) statistics govern the likelihood ratio in planted partition models~\cite{MNS15,Ban18}. There the planted law plays the role of our tilted disorder law.

\subsection{Notation}\label{sec:conventions}
We write $\norm{\cdot}_2$ for the Euclidean norm, $\norm{\cdot}_F$ for the Frobenius norm of a matrix or tensor, and $\langle s,t\rangle_F:=\tr(s^{\intercal}t)$ for the Frobenius inner product. Expectation, variance and probability are always written $\E(\cdot)$, $\var(\cdot)$ and $\pr(\cdot)$. A subscript of $\E$ indicates the variables being averaged or the law used, as in $\E_{\mvgs}(\cdot)$ or $\E_q(\cdot)$. A superscript specifies a law or prescribed profiles. Configuration averages are uniform unless stated otherwise. An undecorated $\E(\cdot)$ is the expectation over the randomness specified by the context. When it averages over the disorder, configurations are held fixed. The same conventions apply to $\pr$ and $\var$. Limits under uniform configurations carry the subscript $0$: $g_0\sim\N(0,\Pi)$ and $K_0\sim\N(0,\Pi\otimes\Pi)$. The letters $g,g',K$ are reserved for the Gibbs limits in Theorem~\ref{thm:C}, and other auxiliary Gaussian variables are denoted by $\zeta$. Constants $C,c$ may depend on $\gk$, $\beta$, the law of $J$ and fixed parameters such as $L$, but not on $N$ or on profiles in a specified window. They may change from one occurrence to the next.

\subsection{Organization}
This article is organized as follows. 
Section~\ref{sec:profile} studies two independent uniform configurations with prescribed profiles.  Theorem~\ref{thm:balanced_clt} is proved in Section~\ref{sec:general-fluctuations}. Section~\ref{sec:general-color} compares profiles and proves Theorem~\ref{thm:full_clt}.  Finally, Section~\ref{sec:general-overlap} proves Theorem~\ref{thm:C}. Appendix~\ref{sec:constants} establishes the entropy and exponential-moment bounds behind both thresholds.

%%%%%%%%%%%%%%%%%%%%%%%%%%%%%%%%%%%%%%%%%%%%%%%%%%%%%%%%%%%%%%%%%%%%%%%%%%%%%%%
\section{Uniform configurations with fixed profiles}\label{sec:profile}
%%%%%%%%%%%%%%%%%%%%%%%%%%%%%%%%%%%%%%%%%%%%%%%%%%%%%%%%%%%%%%%%%%%%%%%%%%%%%%%
All configuration laws in this section are uniform and independent of the disorder. We write $\E_{\mvgs}^{(\mvn)}(\cdot)$ for the uniform average on $\Sigma_N^{(\mvn)}$, and $\E_{\mvgs,\mvgt}^{(\mvn),(\mvn')}(\cdot)$ for independent uniform averages on the indicated classes. The notation extends to three configurations.
After the disorder is averaged out, all quantities in Sections~\ref{sec:general-fluctuations}--\ref{sec:general-overlap} become averages over independent uniform configurations with prescribed profiles. The overlap of two such configurations is a combinatorial statistic of a random permutation, so it has a Gaussian limit (Lemma~\ref{lem:profile-clt}). Its exponential moments are computed by the method of types and controlled by the entropy bound of Lemma~\ref{lem:near-uniform} (Lemma~\ref{lem:profile-moments}).

\subsection{Edge identities and the Gaussian kernel}
Set
\begin{align}
	\xi(x):=\tfrac12\big(\norm{x}_F^2-(\gk-1)^2\big),
	\label{eq:xi-def}
\end{align}
and, for three configurations, define the tensor
\begin{align*}
	R^{(3)}(\mvgs,\mvgt,\mvgt'):=\frac1{\sqrt N}\sum_{i=1}^N
	\phi(\sigma_i)\otimes\phi(\tau_i)\otimes\phi(\tau_i').
\end{align*}

\begin{lem}[Edge identities]\label{lem:edge-id}
	For all configurations $\mvgs,\mvgt$ and all $x\in[\gk]$, and with $\sigma$ uniform on $[\gk]$ in~\eqref{eq:phi-cov},
	\begin{align}
		\phi(\sigma_i)^{\intercal}\phi(\sigma_j)=A_{ij}(\mvgs),\quad
		\norm{\phi(x)}_2^2=\gk-1,
		\label{eq:phi-prod}
	\end{align}
	\begin{align}
		\E_{\sigma}\big(\phi(\sigma)\big)=0,\quad
		\E_{\sigma}\big(\phi(\sigma)\phi(\sigma)^{\intercal}\big)=\Pi,
		\label{eq:phi-cov}
	\end{align}
	\begin{align}
		\sum_{i<j}A_{ij}(\mvgs)=\frac N2\big(\norm{m(\mvgs)}_2^2-\gk+1\big),
		\label{eq:vertex-edge}
	\end{align}
	\begin{align}
		\frac1N\sum_{i<j}A_{ij}(\mvgs)A_{ij}(\mvgt)=\xi\big(\sR(\mvgs,\mvgt)\big),
		\label{eq:pair-edge}
	\end{align}
	\begin{align}
		\frac{\beta^2}{2}\xi\big(\sR(\mvgs,\mvgs)\big)=\frac{\beta^2(\gk-1)}4(N-\gk+1)
								+\frac\alpha2\norm{m(\mvgs)}_2^2.
		\label{eq:selfoverlap}
	\end{align}
	For three configurations,
	\begin{align}
		\frac1N\sum_{i<j}A_{ij}(\mvgs)A_{ij}(\mvgt)A_{ij}(\mvgt')
		=\frac12\big(\norm{R^{(3)}(\mvgs,\mvgt,\mvgt')}_F^2-(\gk-1)^3\big).
		\label{eq:triple-edge}
	\end{align}
	If $N\geq2$ and $\mvgs^1,\mvgs^2,\mvgs^3$ are independent uniform configurations with profiles $\mvn_1,\mvn_2,\mvn_3$, then
	\begin{align}
		\E_{\mvgs^1,\mvgs^2,\mvgs^3}^{(\mvn_1),(\mvn_2),(\mvn_3)}\Big(
		\norm{R^{(3)}(\mvgs^1,\mvgs^2,\mvgs^3)}_F^2\Big)
		=(\gk-1)^3+
		\frac{\prod_{r=1}^3\big(\norm{m_{\mvn_r}}_2^2-\gk+1\big)}{(N-1)^2}.
		\label{eq:triple-overlap-moment}
	\end{align}
\end{lem}

\begin{proof}
	The identities~\eqref{eq:phi-prod} follow by expanding $\phi(x)=\sqrt\gk(e_x-u)$, and~\eqref{eq:phi-cov} follows from $\sum_xe_x=\vone$ and $\sum_xe_xe_x^{\intercal}=I$.
	For the sums over pairs, extend $A_{ij}(\mvgs):=\phi(\sigma_i)^{\intercal}\phi(\sigma_j)$ to all $i,j$, so that $A_{ii}(\mvgs)=\gk-1$.
	Summing~\eqref{eq:phi-prod} over all ordered pairs $(i,j)$ gives
	\begin{align*}
		\sum_{i,j}A_{ij}(\mvgs)=\norm{\textstyle\sum_i\phi(\sigma_i)}_2^2=N\norm{m(\mvgs)}_2^2,
	\end{align*}
	and in the same way
	\begin{align*}
		\sum_{i,j}A_{ij}(\mvgs)A_{ij}(\mvgt)
		=\norm{\textstyle\sum_i\phi(\sigma_i)\phi(\tau_i)^{\intercal}}_F^2
		=N\norm{\sR(\mvgs,\mvgt)}_F^2.
	\end{align*}
	Subtracting the diagonal contributions $N(\gk-1)$ and $N(\gk-1)^2$ and dividing by two proves~\eqref{eq:vertex-edge} and~\eqref{eq:pair-edge}.
	The two possible edge values give $A_{ij}^2=(\gk-2)A_{ij}+\gk-1$.
	Summing this identity over $i<j$ and applying~\eqref{eq:vertex-edge} and~\eqref{eq:pair-edge} gives~\eqref{eq:selfoverlap}.

	For the three-replica identity, the inner product of tensor products gives
	\begin{align*}
		N\norm{R^{(3)}(\mvgs,\mvgt,\mvgt')}_F^2
		=\sum_{i,j}A_{ij}(\mvgs)A_{ij}(\mvgt)A_{ij}(\mvgt').
	\end{align*}
	Subtracting the diagonal contribution $N(\gk-1)^3$ and dividing by $2N$ proves~\eqref{eq:triple-edge}.
	By exchangeability and~\eqref{eq:vertex-edge}, for $i\ne j$,
	\begin{align*}
		\E_{\mvgs}^{(\mvn)}\big(A_{ij}(\mvgs)\big)
		=\frac{\norm{m_{\mvn}}_2^2-\gk+1}{N-1}.
	\end{align*}
	Taking expectations in~\eqref{eq:triple-edge} and using the independence of the three configurations proves~\eqref{eq:triple-overlap-moment}.
\end{proof}

\begin{lem}[The Gaussian kernel]\label{lem:Zsecond}
	Let $K_0\sim\N(0,\Pi\otimes\Pi)$. For $0\leq s<1$, $\E\big(e^{s\xi(K_0)}\big)=e^{\varphi_\gk(s)}$. Moreover, for bounded measurable $f$ and $0<\beta^2<1$,
	\begin{align}
		e^{-\varphi_\gk(\beta^2)}\E\big(e^{\beta^2\xi(K_0)}f(K_0)\big)=\E\left(f\Big(\frac{K_0}{\sqrt{1-\beta^2}}\Big)\right).
		\label{eq:gaussian-tilt}
	\end{align}
\end{lem}

\begin{proof}
	The Gaussian $K_0$ is standard on the range of $\Pi\otimes\Pi$, which has dimension $(\gk-1)^2$. So $\norm{K_0}_F^2$ has the chi-squared law with $(\gk-1)^2$ degrees of freedom, and its moment generating function gives
	\begin{align*}
		\E\big(e^{s\xi(K_0)}\big)=e^{-s(\gk-1)^2/2}(1-s)^{-(\gk-1)^2/2}=e^{\varphi_\gk(s)}.
	\end{align*}
	Multiplying the density of $K_0$ by $e^{\beta^2\norm{K_0}_F^2/2}$ changes its covariance on this range from the identity to $(1-\beta^2)^{-1}$ times the identity. This gives the second assertion.
\end{proof}

\subsection{The profile overlap central limit theorem}

In this subsection we prove a central limit theorem for the overlap of independent uniform configurations whose profiles lie in a bounded magnetization window.
\begin{lem}[Profile-constrained overlaps]\label{lem:profile-clt}
	Fix $L>0$. Let $\mvn_N,\mvn'_N,\mvn''_N$ be profile sequences whose magnetizations have norm at most $L$, and let $\mvgs,\mvgt,\mvgt'$ be independent uniform configurations with these profiles. Then
	\begin{align*}
		\big(\sR(\mvgs,\mvgt),\sR(\mvgs,\mvgt'),\sR(\mvgt,\mvgt')\big)
		\Longrightarrow (K_0,K_0',K_0''),
	\end{align*}
	where $K_0,K_0',K_0''\sim\N(0,\Pi\otimes\Pi)$ are independent. In particular, the limit does not depend on the profiles.
\end{lem}

The proof represents the overlap as a \emph{combinatorial statistic} in the sense of Hoeffding~\cite{Hoe51}. We recall the following central limit theorem.
Call a real array $(c(i,j))_{i,j=1}^N$ \emph{doubly centered} if all of its row sums and all of its column sums vanish.

\begin{lem}[Combinatorial central limit theorem~\cite{Bol84}]\label{lem:bolthausen}
	Let $N\geq2$, let $(c(i,j))_{i,j=1}^{N}$ be a doubly centered real array, let $\pi$ be a uniform random permutation of $[N]$, and set $S := \sum_{i=1}^{N} c(i,\pi(i))$. Then
	\begin{align*}
		\E_\pi(S) = 0\quad\text{and}\quad v := \var_\pi(S) = \frac{1}{N-1}\sum_{i,j} c(i,j)^2.
	\end{align*}
	If $v > 0$, then for an absolute constant $C$ and $\zeta \sim \N(0,1)$,
	\begin{align}
		\sup_{x \in \dR} \Big\lvert \pr_\pi\big( S/\sqrt v \leq x \big) - \pr\big( \zeta \leq x \big) \Big\rvert \leq \frac{C}{N v^{3/2}} \sum_{i,j} \abs{c(i,j)}^3.
		\label{eq:comb-be}
	\end{align}
\end{lem}

\begin{proof}[Proof of Lemma~\ref{lem:profile-clt}]
	The idea is as follows. Conditionally on $\mvgt$, the overlap $\sR(\mvgs,\mvgt)$ is a linear statistic of a uniform permutation, to which Lemma~\ref{lem:bolthausen} applies. Its conditional variance depends on $\mvgt$ only through its profile, up to a small error. For three replicas, the conditional covariance of $\sR(\mvgs,\mvgt)$ and $\sR(\mvgs,\mvgt')$ is governed by $N^{-1/2}\sR(\mvgt,\mvgt')\to0$, which gives asymptotic independence.

	\smallskip\noindent\textbf{Step 1 (permutation representation).} Fix a word $w$ with profile $\mvn_N$ and represent $\sigma_i=w_{\pi(i)}$, where $\pi$ is a uniform permutation independent of $(\mvgt,\mvgt')$. This represents the uniform law on $\Sigma_N^{(\mvn_N)}$, since every configuration of that profile arises from exactly $\prod_xn_x!$ permutations. For a fixed matrix $s$,
	\begin{align*}
		\langle s,\sR(\mvgs,\mvgt)\rangle_F
		=N^{-1/2}\sum_i c(i,\pi(i)),\quad
		c(i,j):=\phi(w_j)^{\intercal}s\phi(\tau_i).
	\end{align*}
	Write $c_{i\cdot},c_{\cdot j},\bar c$ for the row, column, and overall means, and define the doubly centered array
	\begin{align*}
		d(i,j):=c(i,j)-c_{i\cdot}-c_{\cdot j}+\bar c
		= \left(\phi(w_j)- \frac{m_{\mvn_N}}{\sqrt{N}}\right)^{\intercal}s\left(\phi(\tau_i)- \frac{m_{\mvn'_N}}{\sqrt{N}}\right).
	\end{align*}
	Then
	\begin{align*}
		\langle s,\sR(\mvgs,\mvgt)\rangle_F
		=N^{-1/2}\sum_i \left( d(i,\pi(i)) + c_{i \cdot} + c_{\cdot \pi(i)} - \bar c \right)
		=N^{-1/2}\sum_i d(i,\pi(i)) + N^{1/2} \bar c.
	\end{align*}
	Since $N\bar c=m_{\mvn_N}^{\intercal}s\, m_{\mvn'_N}$, the last term is $\cO(N^{-1/2})$ uniformly on the window.

	\smallskip\noindent\textbf{Step 2 (two replicas).}
	Set $a_j :=\phi(w_j)- {m_{\mvn_N}}/{\sqrt{N}}$ and $b_i := \phi(\tau_i)- {m_{\mvn'_N}}/{\sqrt{N}}$. Then
	\begin{align*}
		d(i,j)^2 = (a_j^{\intercal}sb_i)(b_i^{\intercal}s^{\intercal}a_j)
		= \tr(s^{\intercal}a_ja_j^{\intercal}sb_ib_i^{\intercal}).
	\end{align*}
	It follows that
	\begin{align*}
		\frac{1}{N(N-1)}\sum_{i,j} d(i,j)^2 = \frac{N}{N-1} \tr \left[ s^{\intercal}\left(\frac{1}{N} \sum_j a_j a_j^{\intercal}\right)s\left(\frac{1}{N} \sum_i b_i b_i^{\intercal}\right)\right].
	\end{align*}
	Let $p = \mvn_N / N$. Using $p = u + (\gk N)^{-1/2}m_{\mvn_N}$ and $m_{\mvn_N} = \cO(1)$, we get
	\begin{align*}
		\frac{1}{N} \sum_j a_j a_j^{\intercal} = \sum_{x=1}^\gk p_x \phi(x) \phi(x)^{\intercal} - \frac{1}{N} m_{\mvn_N} m_{\mvn_N}^{\intercal}
		\to \Pi.
	\end{align*}
	The same approximation applies to $b_i$, and we obtain the limit of the conditional variance:
	\begin{align*}
		\var\Big( N^{-1/2}\sum_i d(i,\pi(i)) \,\Big|\, \mvgt \Big)=\frac1{N(N-1)}\sum_{i,j}d(i,j)^2
		\longrightarrow \tr(s^{\intercal}\Pi s\Pi)=\norm{\Pi s\Pi}_F^2.
	\end{align*}
	This limit is deterministic, since it depends on $\mvgt$ only through its profile. The arrays are uniformly bounded. If the limiting variance is positive, we apply Lemma~\ref{lem:bolthausen} to the array $d(i,j)/\sqrt N$. Since $\sum_{i,j}\abs{d(i,j)}^3=\cO(N^2)$, the bound~\eqref{eq:comb-be} is $\cO(N^{-1/2})$. If the limiting variance is zero, Chebyshev's inequality applies. Since $\norm{\Pi s\Pi}_F^2=\var(\langle s,K_0\rangle_F)$, the Cram\'er--Wold device proves the two-replica assertion, uniformly along profile sequences in the window.

	\smallskip\noindent\textbf{Step 3 (three replicas).} Fix matrices $s_1,s_2,s_3$ and write
	\begin{align*}
		T_1:=\langle s_1,\sR(\mvgs,\mvgt)\rangle_F,\quad
		T_2:=\langle s_2,\sR(\mvgs,\mvgt')\rangle_F,\quad
		T_3:=\langle s_3,\sR(\mvgt,\mvgt')\rangle_F.
	\end{align*}
	Conditionally on $(\mvgt,\mvgt')$, $T_1+T_2$ is a single permutation statistic:
	\begin{align*}
		T_1 + T_2
		=N^{-1/2}\sum_i c'(i,\pi(i)),
		\quad c'(i,j):=\phi(w_j)^{\intercal}\big(s_1\phi(\tau_i)+s_2\phi(\tau'_i)\big).
	\end{align*}
	As in Steps 1 and 2, define the doubly centered array
	\begin{align*}
		d'(i,j):=c'(i,j)-c'_{i\cdot}-c'_{\cdot j}+\bar c'
		= \left(\phi(w_j)- \frac{m_{\mvn_N}}{\sqrt{N}}\right)^{\intercal} \left[s_1\left(\phi(\tau_i)- \frac{m_{\mvn'_N}}{\sqrt{N}}\right) + s_2\left(\phi(\tau'_i)- \frac{m_{\mvn''_N}}{\sqrt{N}}\right)\right].
	\end{align*}
	Then
	\begin{align*}
		T_1 + T_2
		=N^{-1/2}\sum_i d'(i,\pi(i)) + \cO(N^{-1/2}).
	\end{align*}
	Recall $a_j$ and $b_i$ from Step 2 and set $b_i' := \phi(\tau_i')- {m_{\mvn''_N}}/{\sqrt{N}}$. Then
	\begin{align*}
		d'(i,j)^2
		= \tr(s_1^{\intercal}a_ja_j^{\intercal}s_1b_ib_i^{\intercal})
		+\tr(s_2^{\intercal}a_ja_j^{\intercal}s_2b_i'(b_i')^{\intercal})
		+2\tr(s_2^{\intercal}a_ja_j^{\intercal}s_1b_i(b_i')^{\intercal}).
	\end{align*}
	Moreover, $N^{-1}\sum_i b_i (b_i')^{\intercal} =N^{-1/2}\sR(\mvgt,\mvgt')+ \cO(1/N)$, which tends to zero in probability by the two-replica assertion. Hence
	\begin{align*}
		\var \big(T_1 + T_2 \,\big|\, \mvgt, \mvgt' \big) \longrightarrow \norm{\Pi s_1\Pi}_F^2 + \norm{\Pi s_2\Pi}_F^2\quad\text{in probability}.
	\end{align*}
	If this limit is positive, Lemma~\ref{lem:bolthausen} applies on an event where the conditional variance is bounded below, and the complement of this event has probability tending to zero. The conditional characteristic function of $T_1+T_2$ therefore converges in probability, and in $L^1$ by boundedness, to the deterministic Gaussian characteristic function. If the limit vanishes, the same conclusion follows from the conditional second moment. Since $T_3$ is measurable with respect to $(\mvgt,\mvgt')$,
	\begin{align*}
		\E\big(e^{\mathrm{i}(T_1+T_2+T_3)}\big)
		=\E\Big(e^{\mathrm{i} T_3}\,\E\big(e^{\mathrm{i}(T_1+T_2)}\,\big|\,\mvgt,\mvgt'\big)\Big),
	\end{align*}
	and this factorizes in the limit into the product of the two limiting Gaussian characteristic functions. The two-replica result for $T_3$ and the Cram\'er--Wold device complete the proof.
\end{proof}

\subsection{Type probabilities and exponential moments}

We first estimate pair-type probabilities for uniform configurations with fixed profiles. Combined with the entropy bound of Lemma~\ref{lem:near-uniform}, these estimates give uniform exponential moments of the overlap on bounded magnetization windows and identify their limits.

Throughout the article, $D(q  \|  \pi) := \sum_{x \in \cX} q(x)\ln\frac{q(x)}{\pi(x)} \in [0,\infty]$ denotes the relative entropy of a probability measure $q$ on a finite set $\cX$ with respect to a probability measure $\pi$.
Let $\bar u:=u\otimes u=\gk^{-2}\vone\vone^{\intercal}$ be the uniform law on $[\gk]^2$.

A probability measure $q$ on $[\gk]^2$ is identified with the matrix $(q_{xy})_{x,y\in[\gk]}$, and we write
\begin{align}
	q^{(1)}_x:=\sum_y q_{xy},\quad
	q^{(2)}_y:=\sum_x q_{xy},\quad
	\hat q:=\Pi q\Pi,
	\label{eq:marginals}
\end{align}
for its two marginals and its doubly centered part. Thus $q^{(1)}=q^{(2)}=u$ says exactly that $\gk q$ is doubly stochastic, and in that case $\hat q=q-\bar u$.

\begin{lem}[Profile-constrained type probabilities]\label{lem:profile-types}
	Let $p=\mvn/N$ and $p'=\mvn'/N$. Every realizable pair type $q$ has margins $(q^{(1)},q^{(2)})=(p,p')$, and
	\begin{align}
		\pr_{\mvgs,\mvgt}^{(\mvn),(\mvn')}
			\big(p(\mvgs,\mvgt)=q\big)
		&=\frac{\prod_y\binom{n'_y}{(Nq_{xy})_x}}{\binom N{(n_x)_x}}
		\leq(N+1)^\gk e^{-ND(q\|p\otimes p')}.
		\label{eq:profile-types}
	\end{align}
	On the \emph{core} $\{q:q_{xy}\geq1/(2\gk^2)\text{ for all }x,y\}$ the sharper bound
	\begin{align}
		\pr_{\mvgs,\mvgt}^{(\mvn),(\mvn')}
			\big(p(\mvgs,\mvgt)=q\big)
		\leq C N^{-(\gk-1)^2/2}e^{-ND(q\|p\otimes p')}
		\label{eq:profile-types-sharp}
	\end{align}
	holds.
\end{lem}

\begin{proof}
	Fix $\mvgt$. The configurations $\mvgs\in\Sigma_N^{(\mvn)}$ with pair type $q$ are obtained by choosing, within each color class $\{i:\tau_i=y\}$ of $\mvgt$, which sites receive each color of $\mvgs$, with prescribed counts $(Nq_{xy})_x$. Dividing this count by $\abs{\Sigma_N^{(\mvn)}}=\binom N{(n_x)_x}$ gives the exact ratio. Write $\cS(p):=-\sum_xp_x\ln p_x$. The multinomial bounds of the method of types~\cite[Lemma 2.1.8]{DZ98},
	\begin{align*}
		(n+1)^{-\gk}e^{n\cS(a/n)}\leq\binom n{(a_x)_x}\leq e^{n\cS(a/n)},
	\end{align*}
	and the identity
	\begin{align}
		\sum_y n'_y\cS((Nq_{xy}/n'_y)_x)-N\cS(p)
		&=-N\sum_y p'_y\sum_x\frac{q_{xy}}{p'_y}\ln\frac{q_{xy}}{p'_y}
		+N\sum_{x,y}q_{xy}\ln p_x\nonumber\\
		&=-ND(q\|p\otimes p')
		\label{eq:profile-type-entropy}
	\end{align}
	prove~\eqref{eq:profile-types}.
	On the core, each $Nq_{xy}$ is a positive integer. Stirling's bounds
	\begin{align*}
		\sqrt{2\pi n}\,(n/e)^n\leq n!\leq\sqrt{2\pi n}\,(n/e)^n e^{1/(12n)}
	\end{align*}
	yield
	\begin{align*}
		\binom{n'_y}{(Nq_{xy})_x}
		&\leq e^{1/12}(2\pi)^{-(\gk-1)/2}(n'_y)^{-(\gk-1)/2}
		\left(\prod_x\frac{Nq_{xy}}{n'_y}\right)^{-1/2}
		e^{n'_y\cS((Nq_{xy}/n'_y)_x)}\\
		&\leq C N^{-(\gk-1)/2}e^{n'_y\cS((Nq_{xy}/n'_y)_x)}.
	\end{align*}
	Here we used $Nq_{xy}/n'_y\geq q_{xy}\geq1/(2\gk^2)$ and $n'_y=N\sum_xq_{xy}\geq N/(2\gk)$ on the core. The same computation, with the two Stirling bounds interchanged, bounds the denominator in~\eqref{eq:profile-types} from below:
	\begin{align*}
		\binom N{(n_x)_x}\geq cN^{-(\gk-1)/2}e^{N\cS(p)}.
	\end{align*}
	Combining these bounds with the entropy identity~\eqref{eq:profile-type-entropy} proves~\eqref{eq:profile-types-sharp}.
\end{proof}

For $\gk\geq3$ and uniform margins, the corresponding entropy bound follows from the rowwise argument in~\cite[Lemma~2.6 and Appendix~B]{Kim26}. The following lemma extends the bound to near-uniform margins; its proof is given in Appendix~\ref{sec:constants}.

\begin{lem}[Near-uniform margins]\label{lem:near-uniform}
	Let $q$ be a probability measure on $[\gk]^2$ with margins $(p,p')$. Define $w := q - p\otimes p'$ and $\delta := \norm{p-u}_\infty \vee \norm{p'-u}_\infty$. There is $C = C(\gk)$ such that, if $\gk\delta \leq \tfrac12$, then
	\begin{align}
		D\big(q  \|  p \otimes p'\big)  \geq \frac{\alpha_c(\gk) \gk^2}{2}  \norm{w}_F^2  -  C\big( \delta \norm w_F + \delta^2 \big).
		\label{eq:near-uniform}
	\end{align}
\end{lem}

\begin{lem}[Profile-robust pair moments]\label{lem:profile-moments}
	Let $0<\beta^2<\alpha_c(\gk)$ and $L>0$. Then
	\begin{enumeratei}
		\item
		\begin{align*}
			\sup_N\sup_{\norm{m_{\mvn}}_2,\norm{m_{\mvn'}}_2\leq L}
			\E_{\mvgs,\mvgt}^{(\mvn),(\mvn')}\big(e^{\beta^2\xi(\sR)}\big)
			<\infty,
		\end{align*}
		\item
		\begin{align*}
			\sup_{\norm{m_{\mvn}}_2,\norm{m_{\mvn'}}_2\leq L}
			\left|\E_{\mvgs,\mvgt}^{(\mvn),(\mvn')}\big(e^{\beta^2\xi(\sR)}\big)-e^{\varphi_\gk(\beta^2)}\right|\to0.
		\end{align*}
	\end{enumeratei}
	Here $\sR=\sR(\mvgs,\mvgt)$.
\end{lem}

This is a Laplace-method computation over pair types. The energy of a type is quadratic in its doubly centered part, and Lemma~\ref{lem:near-uniform} shows that the entropy cost dominates it when $\beta^2<\alpha_c(\gk)$. Near $p\otimes p'$ the sharp bound~\eqref{eq:profile-types-sharp} produces a Gaussian lattice sum, whose normalization cancels the prefactor $N^{-(\gk-1)^2/2}$. Types away from the core are exponentially rare. Part (ii) then follows from Lemma~\ref{lem:profile-clt} and uniform integrability.

\begin{proof}
	Write $p=\mvn/N$ and $p'=\mvn'/N$, and let $\cT_N$ be the set of realizable pair types with these margins. Conditioning on the pair type and using $\sR=\gk\sqrt N\hat q$ in the definition of $\xi$, we obtain the type expansion
	\begin{align*}
		\E_{\mvgs,\mvgt}^{(\mvn),(\mvn')}\big(e^{\beta^2\xi(\sR)}\big)
		=e^{-\frac{\beta^2}2(\gk-1)^2}
		\sum_{q\in\cT_N}\pr_{\mvgs,\mvgt}^{(\mvn),(\mvn')}
		\big(p(\mvgs,\mvgt)=q\big)
		e^{\frac{\beta^2\gk^2N}{2}\norm{\hat q}_F^2}.
	\end{align*}

	(i) \smallskip\noindent\textbf{Step 1 (energy and entropy).} Define $w=q-p\otimes p'$ and $\delta=\norm{p-u}_\infty\vee\norm{p'-u}_\infty\leq L/\sqrt{\gk N}$. Since $w$ is doubly centered and $\Pi p=p-u$,
	\begin{align*}
		\hat q-w=\Pi(p\otimes p')\Pi=(p-u)(p'-u)^{\intercal}
		\quad\text{and}\quad
		\norm{\hat q-w}_F\leq\gk\delta^2.
	\end{align*}
	The resulting energy error is controlled by
	\begin{align*}
		\norm{\hat q}_F^2-\norm w_F^2
		=2\langle\hat q-w,w\rangle_F+\norm{\hat q-w}_F^2
		\leq2\gk\delta^2\norm w_F+\gk^2\delta^4.
	\end{align*}
	Set $a:=(\alpha_c(\gk)-\beta^2)\gk^2/2>0$. Combining this estimate with Lemma~\ref{lem:near-uniform} and using $\delta\leq L/\sqrt{\gk N}$, we get, for $N\geq C(1+L^2)$,
	\begin{align*}
		\frac{\beta^2\gk^2N}{2}\norm{\hat q}_F^2-ND(q\|p\otimes p')
		\leq-Na\norm w_F^2+C\sqrt N L\norm w_F+CL^2.
	\end{align*}
	Completing the square,
	\begin{align*}
		-\frac{Na}{2}\norm w_F^2+C\sqrt N L\norm w_F
		=-\frac{Na}{2}\left(\norm w_F-\frac{CL}{a\sqrt N}\right)^2
		+\frac{C^2L^2}{2a}
		\leq CL^2.
	\end{align*}
	Consequently,
	\begin{align}
		\frac{\beta^2\gk^2N}{2}\norm{\hat q}_F^2-ND(q\|p\otimes p')
		\leq-\frac{Na}{2}\norm w_F^2+CL^2.
		\label{eq:profile-exponent}
	\end{align}

	\smallskip\noindent\textbf{Step 2 (lattice sum over core types).} A table with fixed margins is determined by its entries with $x,y<\gk$. Each $Nq_{xy}$ is an integer, and $\norm w_F^2\geq\sum_{x,y<\gk}(q_{xy}-p_xp'_y)^2$. Thus~\eqref{eq:profile-types-sharp} and~\eqref{eq:profile-exponent} give
	\begin{align*}
		\sum_{q\text{ core}}\pr_{\mvgs,\mvgt}^{(\mvn),(\mvn')}
		\big(p(\mvgs,\mvgt)=q\big)e^{\frac{\beta^2\gk^2N}{2}\norm{\hat q}_F^2}
		&\leq C e^{CL^2}N^{-(\gk-1)^2/2}
		\sum_{q\in\cT_N}e^{-\frac{Na}2\norm{q-p\otimes p'}_F^2}\\
		&\leq C e^{CL^2}N^{-(\gk-1)^2/2}
		\prod_{x,y<\gk}\left(\sum_{j\in\dZ}e^{-\frac a{2N}(j-Np_xp'_y)^2}\right)
		\leq Ce^{CL^2}.
	\end{align*}
	For the last inequality, the one-dimensional Gaussian sum satisfies
	\begin{align*}
		\sup_{b\in\dR}\sum_{j\in\dZ}e^{-\frac a{2N}(j-b)^2}
		\leq2+\sqrt{\frac{2\pi N}{a}}\leq C\sqrt N.
	\end{align*}
	There are $(\gk-1)^2$ such factors, so their powers of $N$ cancel the prefactor, uniformly over the profiles.

	\smallskip\noindent\textbf{Step 3 (remaining types).} If $q$ is not in the core, some cell has $q_{xy}<1/(2\gk^2)$. Since $p$ and $p'$ tend to $u$ uniformly on the profile window, $\norm{q-p\otimes p'}_F\geq1/(4\gk^2)$ for $N\geq C(1+L^2)$, after increasing $C$ if necessary. There are at most $(N+1)^{\gk^2}$ types. So~\eqref{eq:profile-types} and Step 1 imply
	\begin{align*}
		\sum_{q\text{ not core}}\pr_{\mvgs,\mvgt}^{(\mvn),(\mvn')}
		\big(p(\mvgs,\mvgt)=q\big)e^{\frac{\beta^2\gk^2N}{2}\norm{\hat q}_F^2}
		\leq e^{CL^2}(N+1)^{\gk+\gk^2}
		e^{-Na/(32\gk^4)}\longrightarrow0,
	\end{align*}
	uniformly over the profiles. In particular, this sum is bounded by $C$ uniformly in $N$. For fixed $L$, the finitely many smaller $N$ also give finite expectations, since their configuration spaces are finite. Thus Steps 2 and 3 prove (i).

	(ii) Choose $\beta^2<s<\alpha_c(\gk)$. Part (i) with $\beta^2$ replaced by $s$ bounds the expectation of $\big(e^{\beta^2\xi(\sR)}\big)^{s/\beta^2}=e^{s\xi(\sR)}$ uniformly over $N$ and the profile window. Hence these exponential weights are uniformly integrable. By Lemma~\ref{lem:profile-clt}, $\sR\Longrightarrow K_0$ along every sequence of profiles in the window. The continuous mapping theorem, uniform integrability, and Lemma~\ref{lem:Zsecond} therefore give
	\begin{align*}
		\E_{\mvgs,\mvgt}^{(\mvn),(\mvn')}\big(e^{\beta^2\xi(\sR)}\big)
		\longrightarrow\E\big(e^{\beta^2\xi(K_0)}\big)
		=e^{-\frac{\beta^2}2(\gk-1)^2}(1-\beta^2)^{-(\gk-1)^2/2}
		=e^{\varphi_\gk(\beta^2)}.
	\end{align*}
	If the uniform convergence in (ii) failed, some sequence of profile pairs in the window would have errors bounded below by a fixed positive number along a subsequence. This contradicts the convergence just proved.
\end{proof}

\begin{rem}[Sharpness of the overlap exponential-moment bound]\label{rem:overlap-threshold}
	For $\gk\geq3$, the entropy bound for uniform margins is attained by the pair measure
	\begin{align*}
		q^*_{xx}=\frac{\gk-1}{\gk^2},\qquad
		q^*_{xy}=\frac1{\gk^2(\gk-1)}\quad(y\ne x).
	\end{align*}
	Each row is $\gk^{-1}$ times a permutation of $p^*$ from Lemma~\ref{lem:eta-phi}, with its large entry on the diagonal. Thus $q^*$ has uniform margins and
	\begin{align*}
		D(q^*\|\bar u)=\frac{\alpha_c(\gk)\gk^2}{2}\norm{q^*-\bar u}_F^2.
	\end{align*}
	Along $N$ divisible by $\gk^2(\gk-1)$, the measure $q^*$ is a realizable pair type for two balanced configurations. The exact probability in~\eqref{eq:profile-types} and Stirling's formula give
	\begin{align*}
		\ln\pr_{\mvgs,\mvgt}^{(\mvn^*),(\mvn^*)}\big(p(\mvgs,\mvgt)=q^*\big)
		=-ND(q^*\|\bar u)+\cO(\ln N).
	\end{align*}
	Consequently, this single type gives the lower bound
	\begin{align*}
		\E_{\mvgs,\mvgt}^{(\mvn^*),(\mvn^*)}\big(e^{\beta^2\xi(\sR)}\big)
		\geq\exp\left\{\frac{N(\beta^2-\alpha_c(\gk))\gk^2}{2}\norm{q^*-\bar u}_F^2-C\ln N\right\}.
	\end{align*}
	For $\beta^2>\alpha_c(\gk)$ this diverges exponentially along the indicated sequence of $N$, so the uniform bound in Lemma~\ref{lem:profile-moments}(i) fails above this threshold. The same conclusion holds for the second moment of the normalized balanced partition function under the truncated disorder: the one-edge expansion~\eqref{eq:general-edge-cumulant} shows that it differs from this overlap exponential moment by factors bounded above and below by positive constants, for each fixed $\beta$.
\end{rem}

%%%%%%%%%%%%%%%%%%%%%%%%%%%%%%%%%%%%%%%%%%%%%%%%%%%%%%%%%%%%%%%%%%%%%%%%%%%%%%%
\section{Balanced free energy for general disorder}\label{sec:general-fluctuations}
%%%%%%%%%%%%%%%%%%%%%%%%%%%%%%%%%%%%%%%%%%%%%%%%%%%%%%%%%%%%%%%%%%%%%%%%%%%%%%%
In this section we prove Theorem~\ref{thm:balanced_clt}. No estimate on the unconstrained free energy is needed. The argument has three parts. First, after a truncation of the disorder, we expand the disorder averages of one, two and three normalized weights (Lemma~\ref{lem:general-kernels}). Together with Section~\ref{sec:profile}, this gives the second moments of the normalized profile partition functions (Lemma~\ref{lem:general-profile-moments}). Second, we prove a joint central limit theorem for the quadratic and loop statistics of an independent edge array (Proposition~\ref{prop:array-joint-clt}). We apply it under the product law and under the tilted law (Proposition~\ref{prop:general-planted-cycles}). Third, the $L^2$ approximation of Section~\ref{sec:ideas} identifies the limit law.

\subsection{Truncation and mixed disorder moments}\label{sec:general-kernels}
Recall the edge notation $E_N$, $A_e$ and $J_e$ from Section~\ref{sec:model}. The exponential moments used below need not exist for the original disorder. We therefore choose $\eps_N\downarrow0$ such that $\eps_N\sqrt N\to\infty$ and
\begin{align}
	\eps_N^{-4}\,\E\big(J^4\ind\{\abs{J}>\eps_N\sqrt N\}\big)\to0.
	\label{eq:general-truncation}
\end{align}
Such a sequence exists since $\E(J^4)<\infty$. For instance, with $\tau(t):=\E\big(J^4\ind\{\abs J>t\}\big)$, which decreases to zero, one may take $\eps_N:=\max\{N^{-1/4},\tau(N^{1/4})^{1/8}\}$. Then $\eps_N\sqrt N\geq N^{1/4}$, and the left side of~\eqref{eq:general-truncation} is at most $\tau(N^{1/4})^{1/2}$.
Set $\wt{J}:=J\ind\{\abs{J}\leq\eps_N\sqrt N\}$ and, on the same probability space as the original disorder, define
\begin{align}
	\wt{J}_{e}:=J_{e}\ind\{\abs{J_{e}}\leq\eps_N\sqrt N\},
	\quad e\in E_N.
	\label{eq:simple-truncation}
\end{align}
The variables $\wt{J}_{e}$ are i.i.d.~copies of $\wt{J}$. Hence they are symmetric, with $\E(\wt{J}_{e})=0$ and $\abs{\wt{J}_{e}}\leq\eps_N\sqrt N$.

By~\eqref{eq:general-truncation} and Markov's inequality, the original and truncated arrays agree with probability $1-o(1)$, and truncation reduces the variance by only $o(N^{-1})$:
\begin{align*}
	N^2\,\pr\big(\abs{J}>\eps_N\sqrt N\big)
	&\leq\eps_N^{-4}\,\E\big(J^4\ind\{\abs{J}>\eps_N\sqrt N\}\big)=o(1),\\
	N\,\E\big(J^2\ind\{\abs{J}>\eps_N\sqrt N\}\big)
	&\leq\eps_N^{-2}\,\E\big(J^4\ind\{\abs{J}>\eps_N\sqrt N\}\big)=o(1).
\end{align*}
In particular, the coupling gives
\begin{align}
	&d_{\mathrm{TV}}\Big(\operatorname{Law}\big((J_e)_{e\in E_N}\big),
		\operatorname{Law}\big((\wt{J}_e)_{e\in E_N}\big)\Big)\nonumber\\
	&\quad\leq\pr\big((J_e)_{e\in E_N}\ne(\wt{J}_e)_{e\in E_N}\big)
	\leq\binom N2\pr\big(\abs{J}>\eps_N\sqrt N\big)=o(1).
	\label{eq:truncation-coupling}
\end{align}
Here $d_{\mathrm{TV}}(P,Q):=\sup_A\abs{P(A)-Q(A)}$. By~\eqref{eq:truncation-coupling}, it suffices to prove the distributional limits of the log-partition functions, and the convergence in probability of Gibbs averages, under the truncated disorder.

The truncated moments satisfy
\begin{align}
	1-\E\big(\wt{J}^2\big)
	&=\E\big(J^2\ind\{\abs{J}>\eps_N\sqrt N\}\big)=o(N^{-1}),
	\label{eq:truncated-second-moment}\\
	\E\big(\wt{J}^4\big)
	&=\E\big(J^4\ind\{\abs{J}\leq\eps_N\sqrt N\}\big)\to\E\big(J^4\big).
	\label{eq:truncated-fourth-moment}
\end{align}
Moreover, for every $M>0$,
\begin{align}
	\E\big(\wt{J}^4\ind\{\abs{\wt{J}}>M\}\big)
	\leq\E\big(J^4\ind\{\abs{J}>M\}\big),
	\label{eq:truncated-fourth-tail}
\end{align}
so the fourth powers are uniformly integrable. We write the truncated Hamiltonian as
\begin{align*}
	\wt{H}_N(\mvgs):=\frac{\beta}{\sqrt N}\sum_{e\in E_N}\wt{J}_{e}A_{e}(\mvgs).
\end{align*}
We keep $\pr$ and $\E$ for the original product probability law and its expectation. Except when the two arrays are compared, all weights and statistics below are built from $\wt J$.

Define
\begin{align*}
	c_N(\mvgs):=\ln\E\big(e^{\wt{H}_N(\mvgs)}\big)
	\quad\text{and}\quad
	e_N(\mvgs):=e^{\wt{H}_N(\mvgs)-c_N(\mvgs)}.
\end{align*}
The quantity $c_N$ is constant on each profile class. Indeed, each edge contributes a term that depends only on whether the edge is monochromatic, and there are $\sum_x\binom{n_x}2$ monochromatic edges for the profile $\mvn$. We write $c_N(\mvn)$ for the common value of $c_N$ on $\Sigma_N^{(\mvn)}$.

For a profile $\mvn$, set
\begin{align*}
	\cZ_N^{(\mvn)}:=\E_{\mvgs}^{(\mvn)}\big( e_N(\mvgs)\big),
\end{align*}
so that $\E\big(\cZ_N^{(\mvn)}\big)=1$. We write $\cZ_N^{\bal}:= \cZ_N^{(\mvn^*)}$ when $\gk\mid N$. These are the disorder-normalized partition functions of the truncated model. By definition,
\begin{align}
	Z_N^{(\mvn)}(\beta;\wt J)=\abs{\Sigma_N^{(\mvn)}}\,e^{ c_N(\mvn)}\, \cZ_N^{(\mvn)}.
	\label{eq:profile-factorization}
\end{align}

The mixed log-moment kernels are
\begin{align}
	\Lambda_N(\mvgs,\mvgt)
	:=\ln\E\big(e_N(\mvgs)e_N(\mvgt)\big)
	\quad\text{and}\quad
	\Lambda^{(3)}_N(\mvgs,\mvgt,\mvgt')
	:=\ln\E\big(e_N(\mvgs)e_N(\mvgt)e_N(\mvgt')\big).
	\label{eq:cumulant-kernels}
\end{align}
The following lemma shows that, under independent uniform configurations with profiles in a fixed bounded magnetization window, the pair kernel is $\beta^2\xi(\sR)$ plus a deterministic constant and the triple kernel is the sum of the three pair kernels, both up to errors tending to zero in probability.

\begin{lem}[Mixed log-moment kernels]\label{lem:general-kernels}
	For fixed $0<\beta^2<1$, uniformly over configurations,
	\begin{align}
		c_N(\mvgs)
		&=\frac{\beta^2(\gk-1)}4(N-\gk+1)+\frac\alpha2\norm{m(\mvgs)}_2^2
			+\frac{\mu_4\beta^4(\gk-1)(\gk^2-3\gk+3)}{48}
			+\fR_N^{(1)}(\mvgs),
		\label{eq:general-self-kernel}\\
		\Lambda_N(\mvgs,\mvgt)
		&=\beta^2\xi\big(\sR(\mvgs,\mvgt)\big)+\varphi_{\gk,J}(\beta^2)-\varphi_{\gk}(\beta^2)
		+\fR_N^{(2)}(\mvgs,\mvgt),
		\label{eq:general-pair-kernel}\\
		\Lambda^{(3)}_N(\mvgs,\mvgt,\mvgt')
		&=\Lambda_N(\mvgs,\mvgt)+\Lambda_N(\mvgs,\mvgt')
			+\Lambda_N(\mvgt,\mvgt')+\fR_N^{(3)}(\mvgs,\mvgt,\mvgt'),
		\label{eq:general-triple-kernel}
	\end{align}
	where
	\begin{align*}
		\abs{\fR_N^{(1)}(\mvgs)}
		\leq o(1)+\frac CN\norm{m(\mvgs)}_2^2.
	\end{align*}
	The remainders $\fR_N^{(2)}$ and $\fR_N^{(3)}$ are bounded uniformly in $N$ and in the configurations. Under independent uniform configurations with profiles in a fixed window $\norm{m_{\mvn}}_2\leq L$, both converge to zero in probability, uniformly over the profiles. Explicitly, for every fixed $L>0$ and $\varepsilon>0$,
	\begin{align*}
		\sup_{\norm{m_{\mvn}}_2,\norm{m_{\mvn'}}_2\leq L}
		\pr_{\mvgs,\mvgt}^{(\mvn),(\mvn')}
		\big(\abs{\fR_N^{(2)}(\mvgs,\mvgt)}>\varepsilon\big)\to0,
	\end{align*}
	and the analogous statement holds for $\fR_N^{(3)}$, with the supremum over three profiles in the same window.
\end{lem}

\begin{proof}
	\emph{The one-edge expansion.} Fix $C>0$. Since $\wt{J}$ is symmetric and $\abs{z\wt{J}}\leq C\eps_N$ for $\abs{z}\leq C/\sqrt N$, Taylor's formula gives, uniformly on this interval,
	\begin{align*}
		\E\big(e^{z\wt{J}}\big) =1+\frac{z^2}{2}\E\big(\wt{J}^2\big)+\frac{z^4}{24}\E\big(\wt{J}^4\big) +\cO\big(\abs{z}^6\E(\abs{\wt{J}}^6)\big).
	\end{align*}
	Here $\E(\abs{\wt{J}}^6)\leq\eps_N^2N\E(J^4)=o(N)$, so the remainder is $o(N^{-2})$. The moment estimates~\eqref{eq:truncated-second-moment} and~\eqref{eq:truncated-fourth-moment} therefore yield
	\begin{align*}
		\E\big(e^{z\wt{J}}\big)=1+\frac{z^2}{2}+\frac{z^4}{24}\E\big(J^4\big)+o(N^{-2}).
	\end{align*}
	The terms after $1$ are $\cO(N^{-1})$. Applying $\ln(1+x)=x-x^2/2+\cO(x^3)$ and recalling $\mu_4=\E(J^4)-3$ gives
	\begin{align}
		\ln\E\big(e^{z\wt{J}}\big)=\frac{z^2}{2}+\frac{\mu_4z^4}{24}+o(N^{-2}),
		\quad\abs{z}\leq C/\sqrt N.
		\label{eq:general-edge-cumulant}
	\end{align}
	All errors are uniform in $z$. Since the edge weights are bounded, we may sum this expansion over the $\binom N2$ edges, for one configuration or for any fixed sum of configurations, with a uniform $o(1)$ error.

	\emph{One configuration.} The two possible values $-1$ and $\gk-1$ of $A_e$ give
	\begin{align*}
		A_e^2=(\gk-2)A_e+\gk-1
		\quad\text{and}\quad
		A_e^4=(\gk-2)(\gk^2-2\gk+2)A_e
			+(\gk-1)(\gk^2-3\gk+3).
	\end{align*}
	By~\eqref{eq:general-edge-cumulant},
	\begin{align*}
		c_N(\mvgs)
		= \sum_{e} \ln \E\big( e^{\beta N^{-1/2} A_e (\mvgs) \wt J_e}\big)
		=\frac{\beta^2}{2N}\sum_e A_e(\mvgs)^2
		+\frac{\mu_4\beta^4}{24N^2}\sum_e A_e(\mvgs)^4+o(1).
	\end{align*}
	By~\eqref{eq:pair-edge} and~\eqref{eq:selfoverlap}, the quadratic term equals the first two terms of~\eqref{eq:general-self-kernel}. The constant part of the quartic term is
	\begin{align*}
		\frac{\mu_4\beta^4(\gk-1)(\gk^2-3\gk+3)}{24N^2}\binom N2
		=\frac{\mu_4\beta^4(\gk-1)(\gk^2-3\gk+3)}{48}+\cO(N^{-1}).
	\end{align*}
	By~\eqref{eq:vertex-edge}, its remaining part is bounded by $CN^{-1}(1+\norm{m(\mvgs)}_2^2)$. This proves~\eqref{eq:general-self-kernel} and the bound on $\fR_N^{(1)}$.

	\emph{Two configurations.} In this proof write $a_e=A_e(\mvgs)$ and $a'_e=A_e(\mvgt)$. Independence of the edges and~\eqref{eq:general-edge-cumulant} give
	\begin{align*}
		\Lambda_N(\mvgs,\mvgt)
		=\frac{\beta^2}{N}\sum_e a_ea'_e +\frac{\mu_4\beta^4}{24N^2}
		\sum_e\big(4a_e^3a'_e+6a_e^2(a'_e)^2+4a_e(a'_e)^3\big)+o(1).
	\end{align*}
	The quadratic term equals $\beta^2\xi\big(\sR(\mvgs,\mvgt)\big)$ by~\eqref{eq:pair-edge}. For the quartic term,
	\begin{align*}
		a_e^3a'_e&=(\gk^2-3\gk+3)a_ea'_e
			+(\gk-1)(\gk-2)a'_e,\\
		a_e^2(a'_e)^2&=(\gk-2)^2a_ea'_e
			+(\gk-1)(\gk-2)(a_e+a'_e)+(\gk-1)^2.
	\end{align*}
	Thus only $6a_e^2(a'_e)^2$ supplies a constant term, and its contribution is
	\begin{align*}
		\frac{\mu_4\beta^4}{24N^2}\,6(\gk-1)^2\binom N2
		=\varphi_{\gk,J}(\beta^2)-\varphi_\gk(\beta^2)+\cO(N^{-1}).
	\end{align*}
	All other quartic terms are linear combinations of $\sum_e a_e$, $\sum_e a'_e$, and $\sum_e a_ea'_e$. By~\eqref{eq:vertex-edge}--\eqref{eq:pair-edge} and the definition of $\xi$, and after absorbing the configuration-independent $\cO(N^{-1})$ terms into $o(1)$,
	\begin{align*}
		\abs{\fR_N^{(2)}(\mvgs,\mvgt)}
		\leq o(1)+\frac CN\big(\norm{m(\mvgs)}_2^2+\norm{m(\mvgt)}_2^2
									+\norm{\sR(\mvgs,\mvgt)}_F^2\big).
	\end{align*}
	This proves~\eqref{eq:general-pair-kernel} with the displayed remainder bound. Since $\norm{m}_2^2\leq N(\gk-1)$ and $\norm{\sR}_F^2\leq N(\gk-1)^2$, the remainder is uniformly bounded. On a fixed profile window, Lemma~\ref{lem:profile-clt} gives $\norm{\sR}_F^2=\cO_{\pr}(1)$ uniformly, so $\fR_N^{(2)}\to0$ in probability, uniformly over the profiles.

	\emph{Three configurations.} Define $a''_e=A_e(\mvgt')$. Subtracting the three pair kernels from the logarithm of the triple moment cancels the quadratic terms. The surviving quartic terms in~\eqref{eq:general-edge-cumulant} give
	\begin{align*}
		\fR_N^{(3)}(\mvgs,\mvgt,\mvgt')
		=\frac{\mu_4\beta^4}{2N^2}\sum_e
		\big(a_e^2a'_ea''_e+a_e(a'_e)^2a''_e+a_ea'_e(a''_e)^2\big)+o(1).
	\end{align*}
	This is uniformly bounded, since each edge weight is bounded. To prove convergence to zero, we use
	\begin{align*}
		a_e^2a'_ea''_e+a_e(a'_e)^2a''_e+a_ea'_e(a''_e)^2
		=3(\gk-2)a_ea'_ea''_e+(\gk-1)(a_ea'_e+a_ea''_e+a'_ea''_e).
	\end{align*}
	The identities~\eqref{eq:pair-edge} and~\eqref{eq:triple-edge} now yield
	\begin{align*}
		\fR_N^{(3)}(\mvgs,\mvgt,\mvgt')
		&=\frac{\mu_4\beta^4}{2N}\Big[
		\frac{3(\gk-2)}2\big(\norm{R^{(3)}(\mvgs,\mvgt,\mvgt')}_F^2-(\gk-1)^3\big)\\
		&\quad+(\gk-1)\big(\xi\big(\sR(\mvgs,\mvgt)\big)+\xi\big(\sR(\mvgs,\mvgt')\big)
		+\xi\big(\sR(\mvgt,\mvgt')\big)\big)\Big]+o(1).
	\end{align*}
	Under independent uniform configurations in a fixed profile window, the expectation in~\eqref{eq:triple-overlap-moment} is uniformly bounded. Markov's inequality gives $\norm{R^{(3)}}_F^2=\cO_{\pr}(1)$ uniformly over these profiles, and Lemma~\ref{lem:profile-clt} gives the same bound for the three pair overlaps. The displayed expression therefore tends to zero in probability, uniformly over the profiles.
\end{proof}

The uniform-configuration estimates of Section~\ref{sec:profile} now give the mixed second moments. The second statement below is the key observation of Step~1 in Section~\ref{sec:ideas}: near-balanced profiles share a common random factor.

\begin{lem}[Mixed second moments and profile comparison]\label{lem:general-profile-moments}
	For $0<\beta^2<\alpha_c(\gk)$ and $L>0$,
	\begin{align}
		\sup_{\norm{m_{\mvn}}_2,\norm{m_{\mvn'}}_2\leq L}
		\abs{\E\big( \cZ_N^{(\mvn)} \cZ_N^{(\mvn')}\big)
						-e^{\varphi_{\gk,J}(\beta^2)}}&\to0,
		\label{eq:general-mixed-second}\\
		\sup_{\norm{m_{\mvn}}_2\leq L}
		\E\Big(\big( \cZ_N^{(\mvn)}- \cZ_N^{(\mvn^*)}\big)^2\Big)&\to0.
		\label{eq:general-profile-collapse}
	\end{align}
\end{lem}

\begin{proof}
	By the definitions of $\cZ_N$ and $\Lambda_N$,
	\begin{align*}
		\E\big(\cZ_N^{(\mvn)}\cZ_N^{(\mvn')}\big)
		=\E_{\mvgs,\mvgt}^{(\mvn),(\mvn')}\Big(
		\E\big(e_N(\mvgs)e_N(\mvgt)\big)\Big)
		=\E_{\mvgs,\mvgt}^{(\mvn),(\mvn')}\big(e^{\Lambda_N(\mvgs,\mvgt)}\big).
	\end{align*}
	Choose $\beta^2<s<\alpha_c(\gk)$. Lemma~\ref{lem:profile-moments}(i), applied with exponent $s$, gives the uniform integrability of $e^{\beta^2\xi(\sR)}$ over the profile window. By Lemma~\ref{lem:general-kernels}, $\fR_N^{(2)}$ is uniformly bounded and converges to zero in probability, uniformly over the profiles. Hence
	\begin{align*}
		\E_{\mvgs,\mvgt}^{(\mvn),(\mvn')}\Big(
		e^{\beta^2\xi(\sR)}\,\abs{e^{\fR_N^{(2)}}-1}\Big)\longrightarrow0
	\end{align*}
	uniformly over the window. Combining~\eqref{eq:general-pair-kernel} with Lemma~\ref{lem:profile-moments}(ii), we obtain
	\begin{align*}
		\E_{\mvgs,\mvgt}^{(\mvn),(\mvn')}\big(e^{\Lambda_N}\big)
		=e^{\varphi_{\gk,J}(\beta^2)-\varphi_\gk(\beta^2)}\,
		\E_{\mvgs,\mvgt}^{(\mvn),(\mvn')}\big(e^{\beta^2\xi(\sR)}\big)+o(1)
		\longrightarrow e^{\varphi_{\gk,J}(\beta^2)},
	\end{align*}
	uniformly over the profiles. This proves~\eqref{eq:general-mixed-second}.

	Finally, $\norm{m_{\mvn^*}}_2\to0$, so $\mvn^*$ belongs to the window for all large $N$. Expand the square in~\eqref{eq:general-profile-collapse} and apply~\eqref{eq:general-mixed-second} to the profile pairs $(\mvn,\mvn)$, $(\mvn,\mvn^*)$, and $(\mvn^*,\mvn^*)$. All three moments have the same limit, uniformly in $\mvn$, and their coefficients $1,-2,1$ sum to zero.
\end{proof}

\subsection{Joint fluctuations for independent edge arrays}\label{sec:general-array}

We prove a joint central limit theorem for independent centered edge arrays, and then apply it under the product law and under the tilted law.

Let $\cL_n$ be the set of loops of length $n$ in $E_N$, so that $|\cL_n|=(N)_n/(2n)$ with $(N)_n:=N(N-1)\cdots(N-n+1)$.
For a centered independent triangular array $\cJ_N=(\cJ_{N,e})_{e\in E_N}$, write
\begin{align}
	X_N(\cJ_N):=\frac1N\sum_{e\in E_N}\big(\cJ_{N,e}^2-\E(\cJ_{N,e}^2)\big)
	\quad\text{and}\quad
	Y_{N,n}(\cJ_N):=\frac1{N^{n/2}}\sum_{\gamma\in\cL_n}\prod_{e\in\gamma}\cJ_{N,e}.
	\label{eq:array-cycle-statistics}
\end{align}
The argument indicates the edge array, and expectations are taken under its specified law.

\begin{prop}[Joint fluctuations for independent edge arrays]\label{prop:array-joint-clt}
	Suppose that the edges $\cJ_{N,e}$ are independent and centered, and that
	\begin{align}
		\max_{e\in E_N}\abs{\E(\cJ_{N,e}^2)-1}\to0
		\quad\text{and}\quad
		\frac1{N^2}\sum_{e\in E_N}\var(\cJ_{N,e}^2)\to v\in[0,\infty).
		\label{eq:array-moment-limits}
	\end{align}
	Assume also that their fourth powers are uniformly integrable:
	\begin{align}
		\lim_{M\to\infty}\sup_{N,e}\E\big(\cJ_{N,e}^4\ind\{\abs{\cJ_{N,e}}\geq M\}\big)=0.
		\label{eq:array-fourth-integrability}
	\end{align}
	Then, for fixed $\ell\geq3$,
	\begin{align*}
		\big(X_N(\cJ_N),Y_{N,3}(\cJ_N),\ldots,Y_{N,\ell}(\cJ_N)\big)
		\Longrightarrow \N(0,v)\otimes\bigotimes_{n=3}^{\ell}\N\left(0,\frac1{2n}\right).
	\end{align*}
\end{prop}

The proof uses a multivariate fourth-moment theorem of de Jong type~\cite{deJ90,DP17}. We first recall its framework. Let $X_1, \ldots, X_n$ be independent, and for $\cI\subseteq[n]$ let $\cF_{\cI} := \sigma(X_j : j \in \cI)$. Every square-integrable function $F$ of $X_1,\ldots,X_n$ has a unique \emph{Hoeffding decomposition} $F = \sum_{\cI \subseteq [n]} F_{\cI}$~\cite{Hoe48}. Here $F_{\cI}$ is $\cF_{\cI}$-measurable, and $\E(F_{\cI} \mid \cF_{\cI'}) = 0$ whenever $\cI \not\subseteq \cI'$. Explicitly, $F_{\cI} = \sum_{\cI' \subseteq \cI}(-1)^{\abs{\cI} - \abs{\cI'}}\E(F \mid \cF_{\cI'})$. Components with different index sets are orthogonal. We call $F$ a \emph{degenerate $U$-statistic of order $p$} if $F_{\cI}=0$ unless $\abs{\cI}=p$. Then $\E(F)=0$, and $\E(F \mid \cF_{\cI'}) = 0$ whenever $\abs{\cI'} < p$.

In our application the coordinates $X_j$ are the edge variables. The centered quadratic sum has order $1$, and a sum over loops of length $n$ of products of centered edges has order $n$. An additional truncation of the quadratic summands avoids an eighth-moment assumption.

\begin{lem}[Multivariate fourth-moment theorem~{\cite[Theorem 1.7]{DP17}}]\label{lem:DP}
	Fix $r \geq 2$ and integers $p_1, \ldots, p_r \geq 1$, and let $n_m \to \infty$. For each $m$, let $F_m = (F_m(1), \ldots, F_m(r))$ be a vector of centered, unit-variance random variables in $L^4$, built from independent random elements $X^{(m)}_1, \ldots, X^{(m)}_{n_m}$. Assume that each $F_m(k)$ is a degenerate $U$-statistic of order $p_k$, with Hoeffding decomposition $F_m(k) = \sum_{\abs\cI=p_k} F_{m,\cI}(k)$. Define the maximal influence of the $k$-th coordinate by
	\begin{align*}
		\mathrm{Infl}_k(F_m) := \max_{1 \leq j \leq n_m} \sum_{\abs{\cI} = p_k,\, \cI \ni j} \var\big(F_{m,\cI}(k)\big).
	\end{align*}
	Let $\Sigma \in \dR^{r \times r}$ be positive semi-definite with $\Sigma(k,k) = 1$ for every $k$, and let $\mvg = (\mvg(1), \ldots, \mvg(r)) \sim \N_r(0,\Sigma)$. Assume that, as $m \to \infty$,
	\begin{enumeratei}
	\item the covariance matrix of $F_m$ converges to $\Sigma$,
	\item $\mathrm{Infl}_k(F_m) \to 0$ for every $k$,
	\item $\E\big(F_m(k)^4\big) \to 3$ for every $k$, and
	\item $\E\big(F_m(j)^2 F_m(k)^2\big) \to \E\big(\mvg(j)^2 \mvg(k)^2\big)$ whenever $j \neq k$ and $p_j = p_k$.
	\end{enumeratei}
	Then $F_m \Longrightarrow \mvg$.
\end{lem}

We verify the moment and influence conditions separately for the loop and quadratic coordinates.

\begin{lem}[Loop moments]\label{lem:loop-chaos}
	Assume the hypotheses of Proposition~\ref{prop:array-joint-clt}, and fix $3\leq n,n'\leq\ell$.
	\begin{enumeratei}
		\item The loop statistics are centered, and distinct lengths are orthogonal:
		\begin{align}
			\E\big(Y_{N,n}(\cJ_N)Y_{N,n'}(\cJ_N)\big)=0\quad(n\ne n').
			\label{eq:Wcov}
		\end{align}
		Their variances satisfy
		\begin{align}
			\E\big(Y_{N,n}(\cJ_N)^2\big)\longrightarrow\frac1{2n}.
			\label{eq:loop-variances}
		\end{align}
		\item For a constant $C$ independent of $N$,
		\begin{align}
			\left|\E\big(Y_{N,n}(\cJ_N)^4\big)-3\,\E\big(Y_{N,n}(\cJ_N)^2\big)^2\right|\leq\frac C N.
			\label{eq:fourth}
		\end{align}
		\item The statistic $Y_{N,n}(\cJ_N)$ is a degenerate $U$-statistic of order $n$ with maximal edge influence $\cO(N^{-2})$. The same holds after variance normalization, for large $N$.
	\end{enumeratei}
\end{lem}

\begin{proof}
	(i) Two distinct loops have an edge that lies in only one of them. Independence and centering of the edges therefore make their monomials orthogonal, which proves~\eqref{eq:Wcov}. Moreover,
	\begin{align*}
		\E\big(Y_{N,n}(\cJ_N)^2\big)=\frac1{N^n}\sum_{\gamma\in\cL_n}\prod_{e\in\gamma}\E\big(\cJ_{N,e}^2\big)
		\longrightarrow\frac1{2n},
	\end{align*}
	since $\abs{\cL_n}=(N)_n/(2n)$ and the edge variances converge uniformly to $1$.

	(ii) In this proof write $Y:=Y_{N,n}(\cJ_N)$, $\tau_e:=\cJ_{N,e}/\sqrt N$, and $\tau^\gamma:=\prod_{e\in\gamma}\tau_e$. Uniform integrability implies $\sup_{N,e}\E(\cJ_{N,e}^4)<\infty$, so $\E(\abs{\tau_e}^p)\leq C N^{-p/2}$ for $2\leq p\leq4$. We expand over quadruples $\mvq=(\gamma_1,\ldots,\gamma_4)\in\cL_n^4$:
	\begin{align*}
		\E(Y^4)-3\,\E(Y^2)^2=\sum_{\mvq\in\cL_n^4}\cK(\mvq),\quad
		\cK(\mvq):=\E\Big(\prod_{i=1}^4\tau^{\gamma_i}\Big)
		-\sum_{\{\{i,j\},\{k,l\}\}}\E\big(\tau^{\gamma_i}\tau^{\gamma_j}\big)\,\E\big(\tau^{\gamma_k}\tau^{\gamma_l}\big),
	\end{align*}
	where the inner sum runs over the three pairings of $[4]$.
	Let $\Gamma := \gamma_1 \cup \cdots \cup \gamma_4$, and let $\nu_e$ be the number of loops $\gamma_i$ that contain $e$. Then $\sum_{e \in \Gamma} \nu_e = 4n$ and $1 \leq \nu_e \leq 4$. By independence, $\E\big(\prod_i \tau^{\gamma_i}\big) = \prod_{e \in \Gamma}\E(\tau_e^{\nu_e})$. If some $\nu_e=1$, the joint expectation vanishes. Each pairing term vanishes as well, since a nonzero pairing term requires two pairs of identical loops, and then every edge has even multiplicity. Hence $\cK(\mvq) = 0$ unless $\nu_e \geq 2$ for all $e$. In that case the joint term and each nonzero pairing term are bounded in absolute value by
	\begin{align*}
		\prod_{e \in \Gamma} C N^{-\nu_e/2} = C^{\abs\Gamma} N^{-2n} \leq C(n) N^{-2n},
	\end{align*}
	so $\abs{\cK(\mvq)} \leq C(n)N^{-2n}$. Since $\nu_e \geq 2$ for all $e$, we have $\abs\Gamma \leq 2n$. Moreover, $\Gamma$ is a union of cycles, so every vertex of $\Gamma$ has degree at least $2$. Hence $\abs{V(\Gamma)} \leq \abs\Gamma \leq 2n$, with equality in the first inequality exactly when $\Gamma$ is a vertex-disjoint union of cycles.

	Suppose that $\abs{V(\Gamma)} = 2n$. Then $\abs\Gamma = 2n$, hence $\nu_e=2$ for every $e\in\Gamma$, and $\Gamma$ is a vertex-disjoint union of cycles. Each $\gamma_i$ is then a whole component of $\Gamma$. So every component of $\Gamma$ is one of the $\gamma_i$ and has length $n$, and $\abs\Gamma = 2n$ leaves exactly two components. In other words, $\mvq$ consists of two vertex-disjoint loops $\gamma \neq \gamma'$, each occurring twice. For such $\mvq$ the joint term is $\E\big((\tau^{\gamma})^2\big) \E\big((\tau^{\gamma'})^2\big)$. Exactly one of the three pairings matches the two equal pairs and contributes the same value, and the other two vanish. Hence $\cK(\mvq) = 0$.

	Only quadruples with $\abs{V(\Gamma)} \leq 2n-1$ therefore contribute. The number of quadruples with a prescribed isomorphism type of $(\gamma_1,\ldots,\gamma_4)$ is at most $C(n)N^{\abs{V(\Gamma)}}$, and the number of types is bounded in terms of $n$. So these quadruples contribute at most $C(n)N^{2n-1}\cdot C(n)N^{-2n} = C(n)/N$, which proves (ii).

	(iii) The normalized loop monomials $N^{-n/2}\prod_{e\in\gamma}\cJ_{N,e}$ are Hoeffding components of order $n$. A fixed edge belongs to $\cO(N^{n-2})$ loops, each with variance $\cO(N^{-n})$, so its influence is $\cO(N^{-2})$. Since the variances have positive limits, this bound persists after normalization.
\end{proof}

For the quadratic coordinate, set
\begin{align}
	\cJ_{N,e}^{(M)}:=\cJ_{N,e}\ind\{\abs{\cJ_{N,e}}<M\}
	\quad\text{and}\quad
	X_N^{(M)}(\cJ_N):=\frac1N\sum_{e\in E_N}\Big((\cJ_{N,e}^{(M)})^2-\E\big((\cJ_{N,e}^{(M)})^2\big)\Big).
	\label{eq:quadratic-fixed-truncation}
\end{align}
This additional cutoff is used only for the quadratic coordinate, and the loop coordinates remain unchanged. It gives a finite fourth moment of $X_N^{(M)}$ without an eighth-moment assumption on the edge variables.

\begin{lem}[Quadratic moments]\label{lem:quadratic-moments}
	Assume the hypotheses of Proposition~\ref{prop:array-joint-clt}. Then $\var\big(X_N(\cJ_N)\big)$ converges to $v$, and for every $M\geq1$,
	\begin{align}
		\E\Big(\abs{X_N(\cJ_N)-X_N^{(M)}(\cJ_N)}^2\Big)
		&\leq\frac1{N^2}\sum_{e\in E_N}\E\big(\cJ_{N,e}^4\ind\{\abs{\cJ_{N,e}}\geq M\}\big),
		\label{eq:quadratic-truncation-error}\\
		\left|\E\big(X_N^{(M)}(\cJ_N)^4\big)-3\var\big(X_N^{(M)}(\cJ_N)\big)^2\right|
		&\leq\frac{CM^4}{N^2},
		\label{eq:quadratic-cutoff-fourth}
	\end{align}
	where $C$ does not depend on $N$ or $M$. The cutoff statistic is a degenerate $U$-statistic of order $1$, with maximal edge influence $\cO(N^{-2})$ uniformly in $M$.
\end{lem}

\begin{proof}
	Independence gives $\var\big(X_N(\cJ_N)\big)=N^{-2}\sum_e\var(\cJ_{N,e}^2)\to v$. The cutoff error is $N^{-1}$ times the centered sum of the independent variables $\cJ_{N,e}^2\ind\{\abs{\cJ_{N,e}}\geq M\}$, which gives~\eqref{eq:quadratic-truncation-error}.

	Write $Q_e:=(\cJ_{N,e}^{(M)})^2-\E\big((\cJ_{N,e}^{(M)})^2\big)$. The uniform fourth-moment bound gives $\E(Q_e^2)\leq C$ and
	$\E(Q_e^4)\leq16\,\E\big((\cJ_{N,e}^{(M)})^8\big)\leq CM^4$. By independence and centering,
	\begin{align*}
		\E\big(X_N^{(M)}(\cJ_N)^4\big)-3\var\big(X_N^{(M)}(\cJ_N)\big)^2
		&=\frac1{N^4}\E\Big(\Big(\sum_e Q_e\Big)^4\Big)-\frac3{N^4}\,\E\Big(\Big(\sum_e Q_e\Big)^2\Big)^2\\
		&=\frac1{N^4}\sum_e\Big(\E(Q_e^4)-3\,\E(Q_e^2)^2\Big)
		=\cO(M^4N^{-2}).
	\end{align*}
	Each $Q_e/N$ is a Hoeffding component of order $1$ with variance at most $C/N^2$, which proves the influence bound.
\end{proof}

\begin{proof}[Proof of Proposition~\ref{prop:array-joint-clt}]
	We may take $\ell\geq4$, since smaller vectors follow by projection. This guarantees at least two loop coordinates, so that Lemma~\ref{lem:DP} (which needs $r\geq2$) applies also when $v=0$. Choose $M_N\to\infty$ with $M_N=o(\sqrt N)$. By Lemma~\ref{lem:quadratic-moments} and~\eqref{eq:array-fourth-integrability},
	\begin{align*}
		X_N(\cJ_N)-X_N^{(M_N)}(\cJ_N)\to0\quad\text{in }L^2
		\quad\text{and}\quad \var\big(X_N^{(M_N)}(\cJ_N)\big)\to v.
	\end{align*}
	If $v>0$, apply Lemma~\ref{lem:DP} to the variance-normalized coordinates of
	\begin{align*}
		\big(X_N^{(M_N)}(\cJ_N),Y_{N,3}(\cJ_N),\ldots,Y_{N,\ell}(\cJ_N)\big).
	\end{align*}
	Their orders $1,3,\ldots,\ell$ are distinct. So the covariance matrix of the normalized vector is the identity, and condition~(iv) is vacuous. Lemmas~\ref{lem:loop-chaos} and~\ref{lem:quadratic-moments} give vanishing influences and fourth moments tending to $3$. For the quadratic coordinate the error is $\cO(M_N^4/N^2)=o(1)$. The resulting independent Gaussian limits and the $L^2$ approximation prove the claim. If $v=0$, then $X_N(\cJ_N)\to0$ in $L^2$, and Lemma~\ref{lem:DP} applies to the loop coordinates alone.
\end{proof}

\subsection{Tilted disorder and mean shifts}\label{sec:general-tilted}

We now verify these assumptions for the two disorder laws. Under the product law, both $(J_e)$ and $(\wt J_e)$ satisfy them with $v=(\E(J^4)-1)/2$, by the assumptions on $J$ and by~\eqref{eq:truncated-second-moment}--\eqref{eq:truncated-fourth-tail}. For the tilted law we first center the edges.

For a fixed configuration $\mvgs\in\Sigma_N$, define the tilted disorder law by
\begin{align}
	\frac{d\pr_N^{\mathrm{tilt},\mvgs}}{d\pr}:=e_N(\mvgs).
	\label{eq:general-planted-law}
\end{align}
This is a probability law since $\E\big(e_N(\mvgs)\big)=1$, and we write $\E^{\mathrm{tilt},\mvgs}(\cdot)$ for its expectation. It is the law of the disorder ``planted'' at $\mvgs$, that is, size-biased by the Boltzmann weight of $\mvgs$. The configuration $\mvgs$ is deterministic throughout. Define the centered edges by
\begin{align*}
	\wt J_e^\circ:=\wt J_e-\E^{\mathrm{tilt},\mvgs}\big(\wt J_e\big).
\end{align*}

\begin{lem}[Tilted edge moments]\label{lem:tilted-edge-moments}
	Fix $\beta>0$ and $\mvgs\in\Sigma_N$. Under $\pr_N^{\mathrm{tilt},\mvgs}$ the edges are independent. Moreover, uniformly over edges and configurations,
	\begin{align}
		\E^{\mathrm{tilt},\mvgs}\big(\wt J_e\big)&=\frac{\beta}{\sqrt N}A_e(\mvgs)+\cO(N^{-3/2}),
		\label{eq:tilted-edge-mean}\\
		\E^{\mathrm{tilt},\mvgs}\big(\wt J_e^2\big)
		&=1+\frac{(\E(J^4)-1)\beta^2}{2N}A_e(\mvgs)^2+o(N^{-1}),
		\label{eq:general-planted-edge-moments}
	\end{align}
	and
	\begin{align*}
		\sup_{\mvgs,e}\left|\E^{\mathrm{tilt},\mvgs}\big(\wt J_e^4\big)-\E(J^4)\right|\to0.
	\end{align*}
	For the centered edges $\wt J_e^\circ$,
	\begin{align*}
		\sup_{\mvgs,e}\left|\E^{\mathrm{tilt},\mvgs}\big((\wt J_e^\circ)^2\big)-1\right|\to0
		\quad\text{and}\quad
		\sup_{\mvgs,e}\left|\E^{\mathrm{tilt},\mvgs}\big((\wt J_e^\circ)^4\big)-\E(J^4)\right|\to0,
	\end{align*}
	and their fourth powers are uniformly integrable:
	\begin{align*}
		\lim_{M\to\infty}
		\sup_{N,\mvgs,e}\E^{\mathrm{tilt},\mvgs}\Big((\wt J_e^\circ)^4\ind\{\abs{\wt J_e^\circ}\geq M\}\Big)=0.
	\end{align*}
\end{lem}

\begin{proof}
	The density factorizes as
	\begin{align}
		e_N(\mvgs)=\prod_{e\in E_N}
		\frac{\exp\{\beta A_e(\mvgs)\wt J_e/\sqrt N\}}
		{\E\big(\exp\{\beta A_e(\mvgs)\wt J/\sqrt N\}\big)},
		\label{eq:general-planted-edges}
	\end{align}
	so the tilted edges are independent. It suffices to estimate the one-edge moment ratios uniformly for $\abs z\leq C/\sqrt N$, and then to substitute $z=\beta A_e(\mvgs)/\sqrt N$. On the truncated support, $\abs{z\wt J}\leq C\eps_N\to0$. Symmetry and Taylor expansion give
	\begin{align*}
		\E\big(e^{z\wt J}\big)&=1+\frac{z^2}{2}\E\big(\wt J^2\big)+\cO(N^{-2}),
		&\E\big(\wt J e^{z\wt J}\big)&=z\E\big(\wt J^2\big)+\cO(N^{-3/2}),\\
		\E\big(\wt J^2e^{z\wt J}\big)&=\E\big(\wt J^2\big)+\frac{z^2}{2}\E\big(\wt J^4\big)+o(N^{-1}),
		&\E\big(\wt J^4e^{z\wt J}\big)&=\E\big(\wt J^4\big)+o(1).
	\end{align*}
	The remainders follow from $\E(\wt J^6)=o(N)$ and $\abs{z\wt J}\leq C\eps_N\to0$. Dividing by $\E\big(e^{z\wt J}\big)$ and using~\eqref{eq:truncated-second-moment} and~\eqref{eq:truncated-fourth-moment}, we obtain
	\begin{align*}
		\frac{\E\big(\wt J e^{z\wt J}\big)}{\E\big(e^{z\wt J}\big)}&=z\E\big(\wt J^2\big)+\cO(N^{-3/2})=z+\cO(N^{-3/2}),\\
		\frac{\E\big(\wt J^2e^{z\wt J}\big)}{\E\big(e^{z\wt J}\big)}
		&=\E\big(\wt J^2\big)+\frac{z^2}{2}\Big(\E\big(\wt J^4\big)-\E\big(\wt J^2\big)^2\Big)+o(N^{-1})
		=1+\frac{\E(J^4)-1}{2}z^2+o(N^{-1}).
	\end{align*}
	This proves~\eqref{eq:tilted-edge-mean}--\eqref{eq:general-planted-edge-moments}. The tilted fourth moment also converges uniformly to $\E(J^4)$ by~\eqref{eq:truncated-fourth-moment}. Write $\mu_e:=\E^{\mathrm{tilt},\mvgs}(\wt J_e)=\cO(N^{-1/2})$. Expanding the powers and using H\"older's inequality with the uniform fourth-moment bound, we get
	\begin{align*}
		\E^{\mathrm{tilt},\mvgs}\big((\wt J_e^\circ)^2\big)-\E^{\mathrm{tilt},\mvgs}\big(\wt J_e^2\big)
		&=-\mu_e^2=\cO(N^{-1}),\\
		\left|\E^{\mathrm{tilt},\mvgs}\big((\wt J_e^\circ)^4\big)-\E^{\mathrm{tilt},\mvgs}\big(\wt J_e^4\big)\right|
		&\leq C\big(\abs{\mu_e}+\mu_e^2+\mu_e^4\big)=\cO(N^{-1/2}).
	\end{align*}
	Thus centering preserves both moment limits, uniformly over edges and configurations.

	For large $M$, the event $\abs{\wt J_e^\circ}\geq M$ implies $\abs{\wt J_e}\geq M/2$ and $\abs{\wt J_e^\circ}\leq2\abs{\wt J_e}$, since the centering constants are uniformly bounded. The one-edge density ratios in~\eqref{eq:general-planted-edges} are uniformly bounded, so~\eqref{eq:truncated-fourth-tail} gives
	\begin{align*}
		\E^{\mathrm{tilt},\mvgs}\Big((\wt J_e^\circ)^4\ind\{\abs{\wt J_e^\circ}\geq M\}\Big)
		\leq C\,\E\big(J^4\ind\{\abs J\geq M/2\}\big)\to0.
	\end{align*}
	This proves uniform integrability, uniformly in $N$, the edge, and the configuration.
\end{proof}

In the applications below, we abbreviate
\begin{align}
	\wt X_N:=X_N(\wt J)
	\quad\text{and}\quad
	\wt Y_{N,n}:=Y_{N,n}(\wt J),
	\label{eq:truncated-loop-statistics}
\end{align}
where the centering in $X_N(\wt J)$ is under the original product law and remains fixed when the law is tilted. For a fixed configuration $\mvgs$, we also write
\begin{align}
	\wt X_N^\circ&:=X_N(\wt J^\circ)
	=\frac1N\sum_{e\in E_N}\Big((\wt J_e^\circ)^2-\E^{\mathrm{tilt},\mvgs}\big((\wt J_e^\circ)^2\big)\Big)
	\quad\text{and}\quad
	\wt Y_{N,n}^\circ:=Y_{N,n}(\wt J^\circ).
	\label{eq:centered-loop-statistics}
\end{align}
The tilde records the truncation of the disorder, and the circle records the centering of the edges under the tilted law. Thus $\E(\wt X_N)=0$ and $\E^{\mathrm{tilt},\mvgs}(\wt X_N^\circ)=\E^{\mathrm{tilt},\mvgs}(\wt Y_{N,n}^\circ)=0$.

The next lemma computes the mean shifts caused by the tilt. The loop mean shift has a transparent origin. To leading order, each tilted edge has mean $\beta A_e(\mvgs)/\sqrt N$, so the product along a loop has mean $\beta^n N^{-n/2}\prod_{e\in\gamma}A_e(\mvgs)$. Summing over loops produces $\tr(\Pi^n)=\gk-1$, because $N^{-1}\sum_i\phi(\sigma_i)\phi(\sigma_i)^{\intercal}\to\Pi$ for near-balanced $\mvgs$.

\begin{lem}[Approximations and mean shifts]\label{lem:tilted-approx}
	Fix $\beta>0$ and $\ell\geq3$.
	\begin{enumeratei}
		\item Uniformly in $\mvgs\in\Sigma_N$,
		\begin{align}
			\E^{\mathrm{tilt},\mvgs}\Big(\big(\wt X_N-\E^{\mathrm{tilt},\mvgs}(\wt X_N)-\wt X_N^\circ\big)^2\Big)&\leq C N^{-1},
			\label{eq:quadratic-tilted-shift}\\
			\E^{\mathrm{tilt},\mvgs}\Big(\big(\wt Y_{N,n}-\E^{\mathrm{tilt},\mvgs}(\wt Y_{N,n})-\wt Y_{N,n}^\circ\big)^2\Big)
			&\leq C(\ell)N^{-1}\quad(3\leq n\leq\ell).
			\label{eq:tilted-loop-remainder}
		\end{align}
		\item For $\mvgs=\mvgs_N\in\Sigma_N^{(\mvn_N)}$ with $\sup_N\norm{m_{\mvn_N}}_2<\infty$,
		\begin{align*}
			\E^{\mathrm{tilt},\mvgs}\big(\wt X_N\big)\to\frac{(\gk-1)\beta^2}{4}\big(\E(J^4)-1\big)
			\quad\text{and}\quad
			\E^{\mathrm{tilt},\mvgs}\big(\wt Y_{N,n}\big)\to\frac{(\gk-1)\beta^n}{2n}.
		\end{align*}
	\end{enumeratei}
\end{lem}
\begin{proof}
	(i) For the quadratic statistic, $\wt J_e^2=(\wt J_e^\circ)^2+2\mu_e\wt J_e^\circ+\mu_e^2$ with $\mu_e=\E^{\mathrm{tilt},\mvgs}(\wt J_e)$. Hence
	\begin{align*}
		\wt X_N-\E^{\mathrm{tilt},\mvgs}(\wt X_N)-\wt X_N^\circ=\frac2N\sum_e\mu_e\,\wt J_e^\circ.
	\end{align*}
	Independence and Lemma~\ref{lem:tilted-edge-moments} bound its variance by $CN^{-1}$, which proves~\eqref{eq:quadratic-tilted-shift}.

	For the loop statistic, write $\tau_e:=\wt J_e^\circ/\sqrt N$ and $a_e:=\mu_e/\sqrt N$. Then $\E^{\mathrm{tilt},\mvgs}(\tau_e)=0$, $\E^{\mathrm{tilt},\mvgs}(\tau_e^2)\leq C/N$, and $\abs{a_e}\leq C/N$, uniformly in $e$ and $\mvgs$. Expanding the loop product and subtracting its expectation gives
	\begin{align*}
		\prod_{e\in\gamma}(\tau_e+a_e)-\prod_{e\in\gamma}a_e
		=\sum_{\emptyset\ne A\subseteq\gamma}\left(\prod_{e\in\gamma\setminus A}a_e\right)\prod_{e\in A}\tau_e.
	\end{align*}
	Hence
	\begin{align*}
		\wt Y_{N,n}-\E^{\mathrm{tilt},\mvgs}(\wt Y_{N,n})-\wt Y_{N,n}^\circ
		=\sum_{k=1}^{n-1}\sum_{\abs A=k}
		\left(\sum_{\substack{\gamma\in\cL_n\\A\subseteq\gamma}}\prod_{e\in\gamma\setminus A}a_e\right)\prod_{e\in A}\tau_e,
	\end{align*}
	where the sum over $A$ runs over the $k$-edge subgraphs contained in some loop of length $n$. Define $N_n(A):=\#\{\gamma\in\cL_n:A\subseteq\gamma\}$. Products over distinct edge sets are orthogonal, since an edge lying in exactly one of $A,A'$ contributes a factor $\E^{\mathrm{tilt},\mvgs}(\tau_e)=0$. Also $\E^{\mathrm{tilt},\mvgs}\big((\prod_{e\in A}\tau_e)^2\big)=\prod_{e\in A}\E^{\mathrm{tilt},\mvgs}(\tau_e^2)$. Thus
	\begin{align*}
		\E^{\mathrm{tilt},\mvgs}\Big(\big(\wt Y_{N,n}-\E^{\mathrm{tilt},\mvgs}(\wt Y_{N,n})-\wt Y_{N,n}^\circ\big)^2\Big)
		\leq C(n)\sum_{k=1}^{n-1}N^{-2(n-k)}N^{-k}\sum_{\abs A=k}N_n(A)^2.
	\end{align*}
	Let $A$ be a $k$-edge proper subgraph of a loop of length $n$, consisting of $p\geq1$ vertex-disjoint paths. Then $\abs{V(A)}=k+p$. For fixed $k$ and $p$ there are at most $C(n)N^{k+p}$ such $A$, and $N_n(A)\leq C(n)N^{n-k-p}$ (choose the remaining $n-k-p$ vertices of the loop and a cyclic arrangement). Hence
	\begin{align*}
		\sum_{\abs A=k}N_n(A)^2\leq C(n)\sum_{p=1}^{n-k}N^{k+p}N^{2(n-k-p)}
		\leq C(n)N^{2n-k-1}.
	\end{align*}
	Each term in the preceding second-moment bound is therefore $\cO(N^{-1})$, which proves~\eqref{eq:tilted-loop-remainder}.

	(ii) Summing~\eqref{eq:general-planted-edge-moments} and using~\eqref{eq:truncated-second-moment} gives
	\begin{align*}
		\E^{\mathrm{tilt},\mvgs}\big(\wt X_N\big)=\frac{(\E(J^4)-1)\beta^2}{2N^2}\sum_e A_e(\mvgs)^2+o(1)
		\to\frac{(\gk-1)\beta^2}{4}\big(\E(J^4)-1\big),
	\end{align*}
	since~\eqref{eq:selfoverlap} gives $\sum_e A_e(\mvgs)^2=(\gk-1)N^2/2+\cO(N)$ on a bounded magnetization window.

	For the loop means, independence and~\eqref{eq:tilted-edge-mean} yield
	\begin{align*}
		\E^{\mathrm{tilt},\mvgs}\big(\wt Y_{N,n}\big)
		=\frac{\beta^n}{N^n}\sum_{\gamma\in\cL_n}\prod_{e\in\gamma}A_e(\mvgs)+o(1).
	\end{align*}
	Since $A_{ij}(\mvgs)=\phi(\sigma_i)^{\intercal}\phi(\sigma_j)$ and each loop is counted $2n$ times by its ordered vertex sequences, we obtain, with $i_{n+1}:=i_1$,
	\begin{align*}
		\sum_{\gamma\in\cL_n}\prod_{e\in\gamma}A_e(\mvgs)
		=\frac1{2n}\sum_{\substack{i_1,\ldots,i_n\\\text{all distinct}}}
		\prod_{r=1}^n\phi(\sigma_{i_r})^{\intercal}\phi(\sigma_{i_{r+1}})
		=\frac1{2n}\tr\left[\left(\sum_{i=1}^N\phi(\sigma_i)\phi(\sigma_i)^{\intercal}\right)^n\right]+\cO(N^{n-1}).
	\end{align*}
	For the second equality, expand the matrix power and use
	\begin{align*}
		\tr\big[(v_1v_1^{\intercal})\cdots(v_nv_n^{\intercal})\big]
		=(v_1^{\intercal}v_2)\cdots(v_n^{\intercal}v_1).
	\end{align*}
	The terms with repeated indices number $\cO(N^{n-1})$ and are uniformly bounded, so after division by $N^n$ they contribute $\cO(N^{-1})$. Moreover, $n_x(\mvgs)/N\to1/\gk$, so
	\begin{align*}
		\frac1N\sum_{i=1}^N\phi(\sigma_i)\phi(\sigma_i)^{\intercal}
		=\sum_{x=1}^{\gk}\frac{n_x(\mvgs)}N\phi(x)\phi(x)^{\intercal}
		\to\frac1\gk\sum_{x=1}^{\gk}\phi(x)\phi(x)^{\intercal}=\Pi.
	\end{align*}
	Since $\Pi^n=\Pi$ and $\tr(\Pi)=\gk-1$, it follows that
	$\E^{\mathrm{tilt},\mvgs}(\wt Y_{N,n})\to(\gk-1)\beta^n/(2n)$.
\end{proof}

For fixed $\ell\geq3$, define
\begin{align}
	\wt G_{N,\ell}&:=\frac{(\gk-1)\beta^2}{2} \wt X_N
					+(\gk-1)\sum_{n=3}^{\ell}\beta^n \wt Y_{N,n},
	\label{eq:general-loop-statistic}\\
	\varphi_{\gk,J,\ell}(\beta^2)&:=\frac{(\gk-1)^2\beta^4}{8}\big(\E(J^4)-1\big)
				+(\gk-1)^2\sum_{n=3}^{\ell}\frac{\beta^{2n}}{2n}.
	\label{eq:general-cycle-variance}
\end{align}
For $\beta^2<1$ we have $\varphi_{\gk,J,\ell}(\beta^2)\uparrow\varphi_{\gk,J}(\beta^2)$ as $\ell\to\infty$. The weights in~\eqref{eq:general-loop-statistic} are chosen so that the limiting mean shift of $\wt G_{N,\ell}$ under the tilted law equals its limiting variance under the product law. This is the signature of an exponential (Le Cam-type) change of measure.

\begin{prop}[Loop fluctuations under the two disorder laws]\label{prop:general-planted-cycles}
	Fix $\beta>0$ and $\ell\geq3$. Under the product law of the truncated disorder,
	\begin{align*}
		\big( \wt X_N, \wt Y_{N,3},\ldots, \wt Y_{N,\ell}\big)
		\Longrightarrow
		\N\left(0,\frac{\E(J^4)-1}{2}\right)
			\otimes\bigotimes_{n=3}^{\ell}\N\left(0,\frac1{2n}\right).
	\end{align*}
	Let $(\mvn_N)$ be profiles with $\sup_N\norm{m_{\mvn_N}}_2<\infty$ and let $\mvgs_N\in\Sigma_N^{(\mvn_N)}$. Under $\pr_N^{\mathrm{tilt},\mvgs_N}$, the same vector converges to a Gaussian vector with independent coordinates, the same variances, and respective means
	\begin{align*}
		\frac{(\gk-1)\beta^2}{4}\big(\E(J^4)-1\big)
		\quad\text{and}\quad \frac{(\gk-1)\beta^n}{2n},\quad 3\leq n\leq\ell.
	\end{align*}
	Consequently, under the product law and under the tilted law, respectively,
	\begin{align*}
		\wt G_{N,\ell}\Longrightarrow\N\big(0,\varphi_{\gk,J,\ell}(\beta^2)\big)
		\quad\text{and}\quad
		\wt G_{N,\ell}\Longrightarrow
			\N\big(\varphi_{\gk,J,\ell}(\beta^2),\varphi_{\gk,J,\ell}(\beta^2)\big).
	\end{align*}
\end{prop}

\begin{proof}
	Under the product law, the edges $\wt J_e$ are independent and centered by symmetry. By~\eqref{eq:truncated-second-moment} and~\eqref{eq:truncated-fourth-moment},
	\begin{align*}
		\max_{e\in E_N}\abs{\E\big(\wt J_e^2\big)-1}\to0
		\quad\text{and}\quad \frac1{N^2}\sum_{e\in E_N}\var\big(\wt J_e^2\big)
		=\frac{\binom N2}{N^2}\Big(\E\big(\wt J^4\big)-\E\big(\wt J^2\big)^2\Big)
		\to\frac{\E(J^4)-1}{2}.
	\end{align*}
	Moreover,~\eqref{eq:truncated-fourth-tail} gives the uniform integrability of their fourth powers. Thus all assumptions of Proposition~\ref{prop:array-joint-clt} hold with $\cJ_{N,e}=\wt J_e$ and $v=(\E(J^4)-1)/2$, which proves the first assertion.

	For the tilted law, Lemma~\ref{lem:tilted-edge-moments} verifies the assumptions of Proposition~\ref{prop:array-joint-clt} for $\cJ_{N,e}=\wt J_e^\circ$, with $v=(\E(J^4)-1)/2$. Thus $(\wt X_N^\circ,\wt Y_{N,3}^\circ,\ldots,\wt Y_{N,\ell}^\circ)$ has the centered Gaussian limit of the first assertion. By Lemma~\ref{lem:tilted-approx}, the expectation of any bounded Lipschitz function of the original vector differs by $o(1)$ from that of the centered vector translated by the limiting means. This proves the tilted assertion. The claims for $\wt G_{N,\ell}$ follow from its definition, since
	\begin{align*}
		\frac{(\gk-1)\beta^2}{2}\cdot\frac{(\gk-1)\beta^2}{4}\big(\E(J^4)-1\big)
		+(\gk-1)\sum_{n=3}^{\ell}\beta^n\frac{(\gk-1)\beta^n}{2n}
		=\varphi_{\gk,J,\ell}(\beta^2)
	\end{align*}
	for the mean, and the variance is computed in the same way.
\end{proof}

\subsection{Proof of the balanced free-energy CLT}\label{sec:general-constrained}

\begin{proof}[Proof of Theorem~\ref{thm:balanced_clt}]
	We show that $\cZ_N^{(\mvn_N)}$ is close in $L^2$ to an explicit function of the loop statistics, then take logarithms. Fix a sequence $(\mvn_N)$ in a bounded magnetization window. Lemma~\ref{lem:general-profile-moments} gives $\E\big( (\cZ_N^{(\mvn_N)})^2\big)\to e^{\varphi_{\gk,J}(\beta^2)}$.
	We show that the loop statistics account for this entire second moment. Fix $\ell\geq3$ and $B>0$, and define the bounded continuous function
	\begin{align*}
		f_{\ell,B}(x):=\min\big\{e^{x-\varphi_{\gk,J,\ell}(\beta^2)/2},B\big\}.
	\end{align*}
	Our goal is the $L^2$ approximation
	\begin{align}
		\lim_{\ell\to\infty}\lim_{B\to\infty}\limsup_{N\to\infty}
		\E\Big(\big(\cZ_N^{(\mvn_N)}-f_{\ell,B}(\wt G_{N,\ell})\big)^2\Big)=0.
		\label{eq:general-lognormal-approximation}
	\end{align}
	By~\eqref{eq:general-planted-law},
	\begin{align*}
		\E\Big(\cZ_N^{(\mvn_N)}f_{\ell,B}(\wt G_{N,\ell})\Big)
		=\E_{\mvgs}^{(\mvn_N)}\Big(\E^{\mathrm{tilt},\mvgs}\big(f_{\ell,B}(\wt G_{N,\ell})\big)\Big).
	\end{align*}
	The inner expectation is the same for all configurations of this profile, because the product law and $\wt G_{N,\ell}$ are invariant under permutations of the vertices. Thus Proposition~\ref{prop:general-planted-cycles} gives
	\begin{align*}
		\lim_{N\to\infty}\E\Big(\big( \cZ_N^{(\mvn_N)}-f_{\ell,B}( \wt G_{N,\ell})\big)^2\Big)
		=e^{\varphi_{\gk,J}(\beta^2)}-2\E\big(f_{\ell,B}(\zeta'_{\ell})\big)+\E\big(f_{\ell,B}(\zeta_{\ell})^2\big),
	\end{align*}
	where $\zeta_{\ell}\sim\N(0,\varphi_{\gk,J,\ell}(\beta^2))$ and $\zeta'_{\ell}\sim\N(\varphi_{\gk,J,\ell}(\beta^2),\varphi_{\gk,J,\ell}(\beta^2))$. The density of the law of $\zeta'_\ell$ with respect to that of $\zeta_\ell$ is $e^{x-\varphi_{\gk,J,\ell}(\beta^2)/2}$. Letting $B\to\infty$, monotone convergence turns the right-hand side into
	\begin{align*}
		e^{\varphi_{\gk,J}(\beta^2)}-2e^{\varphi_{\gk,J,\ell}(\beta^2)}
			+e^{\varphi_{\gk,J,\ell}(\beta^2)}
		=e^{\varphi_{\gk,J}(\beta^2)}-e^{\varphi_{\gk,J,\ell}(\beta^2)},
	\end{align*}
	which tends to zero as $\ell\to\infty$. This proves~\eqref{eq:general-lognormal-approximation}. For a bounded Lipschitz function $h$, the Cauchy--Schwarz inequality gives
	\begin{align*}
		\left|\E\big(h( \cZ_N^{(\mvn_N)})\big)-\E\big(h(f_{\ell,B}( \wt G_{N,\ell}))\big)\right|
		\leq\operatorname{Lip}(h)
			\left(\E\Big(\big( \cZ_N^{(\mvn_N)}
				-f_{\ell,B}( \wt G_{N,\ell})\big)^2\Big)\right)^{1/2}.
	\end{align*}
	Let first $N\to\infty$, then $B\to\infty$, and finally $\ell\to\infty$. The Gaussian approximants converge in law to $e^{\zeta-\varphi_{\gk,J}(\beta^2)/2}$ with $\zeta\sim\N(0,\varphi_{\gk,J}(\beta^2))$. This proves the lognormal limit of $\cZ_N^{(\mvn_N)}$. Since the limit is strictly positive, the continuous mapping theorem gives
	\begin{align}
		\ln \cZ_N^{(\mvn_N)}
			\Longrightarrow\N\left(-\frac12\varphi_{\gk,J}(\beta^2),
										\varphi_{\gk,J}(\beta^2)\right).
		\label{eq:general-profile-lognormal}
	\end{align}
	No estimate on negative moments of $\cZ_N^{(\mvn_N)}$ is needed.

	Finally, take logarithms in~\eqref{eq:profile-factorization} and substitute~\eqref{eq:general-self-kernel} and~\eqref{eq:general-variance}. Since $\fR_N^{(1)}\to0$ on the window, the limiting mean correction is
	\begin{align*}
		\frac{\mu_4\beta^4(\gk-1)(\gk^2-3\gk+3)}{48}-\frac12\varphi_{\gk,J}(\beta^2)
		=-\frac12\varphi_{\gk}(\beta^2)
			+\frac{\mu_4\beta^4(\gk-1)(\gk^2-6\gk+6)}{48}.
	\end{align*}
	The coupling~\eqref{eq:truncation-coupling} transfers the distributional statement to the original disorder.
\end{proof}

%%%%%%%%%%%%%%%%%%%%%%%%%%%%%%%%%%%%%%%%%%%%%%%%%%%%%%%%%%%%%%%%%%%%%%%%%%%%%%%
\section{Profile comparison and unconstrained free energy}\label{sec:general-color}
%%%%%%%%%%%%%%%%%%%%%%%%%%%%%%%%%%%%%%%%%%%%%%%%%%%%%%%%%%%%%%%%%%%%%%%%%%%%%%%
We now remove the profile constraint. By~\eqref{eq:profile-factorization}, the unconstrained partition function is a sum over profiles of deterministic weights times the random factors $\cZ_N^{(\mvn)}$. On bounded windows these factors are all close to the reference factor $\cZ_N^{(\mvn^*)}$ by Lemma~\ref{lem:general-profile-moments}. Outside such windows the deterministic Curie--Weiss--Potts weights are small when $\alpha<\alpha_c(\gk)$. Theorem~\ref{thm:balanced_clt} provides the lower-tail control that makes the comparison relative, namely $\cZ_N^{(\mvn^*)}$ is bounded away from zero in probability. All disorder-dependent quantities are built from $\wt{J}$ until the final transfer. 

\subsection{Profile equivalence and free-energy comparison}
Let $P_{\mvn}:=\abs{\Sigma_N^{(\mvn)}}/\gk^N$ be the multinomial probability of the profile $\mvn$ under independent uniform colors, and define
\begin{align}
	\widehat P_{\mvn}:=\frac{P_{\mvn}e^{c_N(\mvn)}}
		{\sum_{\mvn'}P_{\mvn'}e^{c_N(\mvn')}}
	\quad\text{and}\quad
	w_{\mvn}:=e^{c_N(\mvn)-c_N(\mvn^*)}.
	\label{eq:tilted-profiles}
\end{align}

Here $c_N(\mvn)$ is the common value of $c_N$ on $\Sigma_N^{(\mvn)}$, as in~\eqref{eq:profile-factorization}. These weights are deterministic. The law $\widehat P$ is the annealed profile law, and $w_{\mvn}$ is the annealed Boltzmann factor $e^{c_N}$ of the profile $\mvn$ relative to that of $\mvn^*$. By~\eqref{eq:profile-factorization}, all of the disorder sits in the factors $\cZ_N^{(\mvn)}$.

For $0\leq\gamma<\alpha_c(\gk)$, Lemmas~\ref{lem:expmoment} and~\ref{lem:eta-phi} in Appendix~\ref{sec:constants}, applied to $W=\sqrt\gamma\,\phi(\sigma_1)$ with $\sigma_1$ uniform on $[\gk]$, give
\begin{align*}
	\sup_N\sum_{\mvn}P_{\mvn}e^{\gamma\norm{m_{\mvn}}_2^2/2}
	=\sup_N\E_{\mvgs}\big(e^{\gamma\norm{m(\mvgs)}_2^2/2}\big)
	\leq\big(1-\gamma\eta_\phi(\gk)\big)^{-\gk/2}<\infty.
\end{align*}
This bound controls the contribution of profiles with large magnetization in the comparison below.

\begin{prop}[Profile equivalence and free-energy comparison]\label{prop:profile-equiv}
	Assume that $0<\beta^2<\alpha_c(\gk)$ and $\alpha<\alpha_c(\gk)$. For the original disorder $J$,
	\begin{align}
		\sum_{\mvn}\left|\big\langle\ind\{n(\mvgs)=\mvn\}\big\rangle_{\varrho_{N,\beta}}
						-\widehat P_{\mvn}\right|&\longrightarrow0
						\quad\text{in probability, and}
		\label{eq:profile-total-variation}\\
		\ln Z_N(\beta;J)-\ln Z_N^{(\mvn^*)}(\beta;J)
			-\ln\frac{\gk^N}{\abs{\Sigma_N^{(\mvn^*)}}}
		&\longrightarrow-\frac{\gk-1}{2}\ln(1-\alpha)
						\quad\text{in probability}.
		\label{eq:transfer-full}
	\end{align}
\end{prop}

\begin{proof}
	Throughout, $\sum_{\leq L}$ and $\sum_{>L}$ denote the profile sums restricted to $\norm{m_{\mvn}}_2\leq L$ and to $\norm{m_{\mvn}}_2>L$, for an arbitrary $L>0$.
	By~\eqref{eq:general-self-kernel} and $\norm{m_{\mvn^*}}_2^2=\cO(N^{-1})$,
	\begin{align*}
		\ln w_{\mvn}=\frac\alpha2\norm{m_{\mvn}}_2^2
		+o(1)+\cO\left(\frac{\norm{m_{\mvn}}_2^2}{N}\right),
	\end{align*}
	uniformly in the profile.
	Fix $\alpha<\gamma<\gamma'<\alpha_c(\gk)$. The expansion gives $w_{\mvn}\leq C e^{\gamma\norm{m_{\mvn}}_2^2/2}$ for all sufficiently large $N$, and $w_{\mvn}-e^{\alpha\norm{m_{\mvn}}_2^2/2}\to0$ uniformly on every fixed window.

	\emph{Profile equivalence.} Set
	\begin{align*}
		\mathrm{Den}_N:=\sum_{\mvn}P_{\mvn}w_{\mvn}\cZ_N^{(\mvn)}.
	\end{align*}
	The Gibbs profile probabilities are $P_{\mvn}w_{\mvn}\cZ_N^{(\mvn)}/\mathrm{Den}_N$, whereas $\widehat P_{\mvn}=P_{\mvn}w_{\mvn}/\sum_{\mvn'}P_{\mvn'}w_{\mvn'}$. We use the identity
	\begin{align*}
		\widehat P_{\mvn}
		-\frac{P_{\mvn}w_{\mvn}\cZ_N^{(\mvn^*)}}{\mathrm{Den}_N}
		=\frac{\widehat P_{\mvn}}{\mathrm{Den}_N}
		\sum_{\mvn'}P_{\mvn'}w_{\mvn'}
		\big(\cZ_N^{(\mvn')}-\cZ_N^{(\mvn^*)}\big).
	\end{align*}
	Adding and subtracting $P_{\mvn}w_{\mvn}\cZ_N^{(\mvn^*)}/\mathrm{Den}_N$ therefore gives
	\begin{align}
		\sum_{\mvn}\left|\frac{P_{\mvn}w_{\mvn}\cZ_N^{(\mvn)}}{\mathrm{Den}_N}
		-\widehat P_{\mvn}\right|
		\leq 2\,\frac{\sum_{\mvn}P_{\mvn}w_{\mvn}
		\abs{\cZ_N^{(\mvn)}-\cZ_N^{(\mvn^*)}}}{\mathrm{Den}_N}.
		\label{eq:profile-tv-bound}
	\end{align}

	We first estimate the numerator of~\eqref{eq:profile-tv-bound}. Since $\E\big(\cZ_N^{(\mvn)}\big)=1$,
	\begin{align}
		\E\Big(\sum_{\mvn}P_{\mvn}w_{\mvn}
		\abs{\cZ_N^{(\mvn)}-\cZ_N^{(\mvn^*)}}\Big)
		&\leq\sum_{\leq L}P_{\mvn}w_{\mvn}\,
		\E\big(\abs{\cZ_N^{(\mvn)}-\cZ_N^{(\mvn^*)}}\big)
		+2\sum_{>L}P_{\mvn}w_{\mvn}\nonumber\\
		&\leq C e^{\gamma L^2/2}\sum_{\leq L}P_{\mvn}\,
		\E\big(\abs{\cZ_N^{(\mvn)}-\cZ_N^{(\mvn^*)}}\big)
		+2\sum_{>L}P_{\mvn}w_{\mvn}.
		\label{eq:general-profile-split}
	\end{align}
	Since $\sum_{\leq L}P_{\mvn}\leq1$, the first term is at most
	\begin{align*}
		C e^{\gamma L^2/2}\sup_{\norm{m_{\mvn}}_2\leq L}
		\left(\E\Big(\big(\cZ_N^{(\mvn)}-\cZ_N^{(\mvn^*)}\big)^2\Big)\right)^{1/2},
	\end{align*}
	which tends to $0$ by~\eqref{eq:general-profile-collapse}. For the second term,
	\begin{align}
		\sum_{>L}P_{\mvn}w_{\mvn}
		&\leq C \sum_{>L}P_{\mvn} e^{\gamma\norm{m_{\mvn}}_2^2/2}\nonumber\\
		&\leq C e^{-(\gamma'-\gamma)L^2/2}\, \E_{\mvgs}\big(e^{\gamma'\norm{m(\mvgs)}_2^2/2}\big)
		\leq C\big(1-\gamma'\eta_\phi(\gk)\big)^{-\gk/2}
		e^{-(\gamma'-\gamma)L^2/2}.
		\label{eq:general-profile-tail}
	\end{align}
	In the last inequality we applied Lemma~\ref{lem:expmoment} to $W:=\sqrt{\gamma'}\phi(\sigma_1)$, whose optimal sub-Gaussian constant is $\gamma'\eta_\phi(\gk)<1$ by Lemma~\ref{lem:eta-phi}. Letting $N\to\infty$ in~\eqref{eq:general-profile-split}, we get
	\begin{align*}
		\limsup_{N\to\infty}\E\Big(\sum_{\mvn}P_{\mvn}w_{\mvn}
		\abs{\cZ_N^{(\mvn)}-\cZ_N^{(\mvn^*)}}\Big)
		\leq C e^{-(\gamma'-\gamma)L^2/2}.
	\end{align*}
	Since $L$ is arbitrary,
	\begin{align}
		\sum_{\mvn}P_{\mvn}w_{\mvn}
		\abs{\cZ_N^{(\mvn)}-\cZ_N^{(\mvn^*)}}
		\to0\quad\text{in }L^1.
		\label{eq:general-profile-weighted}
	\end{align}

	Next we compute the total annealed weight. The weights $w_{\mvn}$ converge to $e^{\alpha\norm{m_{\mvn}}_2^2/2}$ uniformly on bounded windows, and~\eqref{eq:general-profile-tail} controls both $w_{\mvn}$ and $e^{\alpha\norm{m_{\mvn}}_2^2/2}$ outside them. Hence
	\begin{align*}
		\sum_{\mvn}P_{\mvn} \left(w_{\mvn} - e^{\alpha\norm{m_{\mvn}}_2^2/2}\right) \to 0.
	\end{align*}
	Under the uniform measure, $m(\mvgs)\Longrightarrow g_0\sim\N(0,\Pi)$ by the central limit theorem. Lemma~\ref{lem:expmoment} with $W:=\sqrt\gamma\,\phi(\sigma_1)$ gives $\sup_N\E_{\mvgs}\big(e^{\gamma\norm{m}_2^2/2}\big)<\infty$. Since $\gamma>\alpha$, the variables $e^{\alpha\norm m_2^2/2}$ are uniformly integrable (trivially if $\alpha=0$). Therefore
	\begin{align}
		\sum_{\mvn}P_{\mvn}w_{\mvn}
		\to\E\big(e^{\alpha\norm {g_0}_2^2/2}\big)=(1-\alpha)^{-(\gk-1)/2}.
		\label{eq:general-profile-tilt}
	\end{align}
	The last identity holds because $\norm{g_0}_2^2$ has the chi-squared law with $\gk-1$ degrees of freedom. By~\eqref{eq:general-profile-weighted},
	\begin{align*}
		\mathrm{Den}_N-\cZ_N^{(\mvn^*)}\sum_{\mvn}P_{\mvn}w_{\mvn}
		\to0\quad\text{in }L^1.
	\end{align*}
	By~\eqref{eq:general-profile-lognormal}, $\cZ_N^{(\mvn^*)}$ converges in law to an almost surely positive limit, so $1/\cZ_N^{(\mvn^*)}=\cO_{\pr}(1)$. Dividing by $\cZ_N^{(\mvn^*)}$ yields
	\begin{align*}
		\frac{\mathrm{Den}_N}{\cZ_N^{(\mvn^*)}}
		\to(1-\alpha)^{-(\gk-1)/2}\quad\text{in probability}.
	\end{align*}
	In particular,
	\begin{align}
		(\mathrm{Den}_N)^{-1} =\left(\sum_{\mvn}P_{\mvn}w_{\mvn}\cZ_N^{(\mvn)}\right)^{-1}
		=\cO_{\pr}(1).
		\label{eq:general-gibbs-denominator}
	\end{align}
	Now~\eqref{eq:profile-tv-bound}, \eqref{eq:general-profile-weighted} and~\eqref{eq:general-gibbs-denominator} prove~\eqref{eq:profile-total-variation} for the truncated disorder.

	\emph{Free-energy comparison.} The profile decomposition~\eqref{eq:profile-factorization} gives
	\begin{align*}
		\frac{Z_N(\beta;\wt J)}{Z_N^{(\mvn^*)}(\beta;\wt J)}\cdot
		\frac{\abs{\Sigma_N^{(\mvn^*)}}}{\gk^N}
		=\frac{\mathrm{Den}_N}{\cZ_N^{(\mvn^*)}}
		\to(1-\alpha)^{-(\gk-1)/2}\quad\text{in probability}.
	\end{align*}
	Taking logarithms proves~\eqref{eq:transfer-full}. Finally, the coupling~\eqref{eq:truncation-coupling} transfers both assertions from $\wt J$ to the original disorder.
\end{proof}

\begin{proof}[Proof of Theorem~\ref{thm:full_clt}]
	Apply Theorem~\ref{thm:balanced_clt} to $\mvn^*(N)$ and add~\eqref{eq:transfer-full}.
	The profile cardinalities cancel, and $\norm{m_{\mvn^*}}_2^2\to0$.
	Passing to the centering of the statement subtracts
	$\beta^2(\gk-1)(\gk-2)/4=(\gk-1)\alpha/2$.
	Since
	\begin{align*}
		-\frac{\gk-1}{2}\alpha
		-\frac{\gk-1}{2}\ln(1-\alpha)
		=(\gk-1)\varphi(\alpha),
	\end{align*}
	the stated formula follows.
\end{proof}

\subsection{Color fluctuations}
By Proposition~\ref{prop:profile-equiv}, the quenched profile law is asymptotically the deterministic annealed law $\widehat P$, a Curie--Weiss--Potts law. Its magnetization is therefore asymptotically Gaussian, with the covariance $\Pi$ of the uniform measure inflated by the quadratic tilt.

\begin{lem}[Color fluctuations under general disorder]\label{lem:general-color-clt}
	Let $0<\beta<\beta_{\max}(\gk)$ and $g\sim\N(0,\Pi/(1-\alpha))$. For every bounded continuous $f:\dR^{\gk}\to\dR$,
	\begin{align}
		\big\langle f(m)\big\rangle_{\varrho_{N,\beta}}
			\to\E\big(f(g)\big)\quad\text{in probability}.
		\label{eq:general-color-clt}
	\end{align}
	Moreover, there are $C,c>0$ such that, for every $L>0$,
	\begin{align}
		\big\langle\ind\{\norm{m}_2>L\}\big\rangle_{\varrho_{N,\beta}}
			\leq C e^{-cL^2}+o_{\pr}(1).
		\label{eq:general-color-tail}
	\end{align}
\end{lem}

\begin{proof}
	We first work with the truncated disorder. The total variation comparison~\eqref{eq:profile-total-variation} gives
	\begin{align*}
		\big\langle f(m)\big\rangle_{\varrho_{N,\beta}}
		=\sum_{\mvn}\widehat P_{\mvn}f(m_{\mvn})+o_{\pr}(1)
		=\frac{\sum_{\mvn}P_{\mvn}w_{\mvn}f(m_{\mvn})}{\sum_{\mvn}P_{\mvn}w_{\mvn}}+o_{\pr}(1).
	\end{align*}
	Since $f$ is bounded, the tail estimate~\eqref{eq:general-profile-tail} and the uniform convergence $w_{\mvn}-e^{\alpha\norm{m_{\mvn}}_2^2/2}\to0$ on each fixed window imply
	\begin{align*}
		\frac{\sum_{\mvn}P_{\mvn}w_{\mvn}f(m_{\mvn})}{\sum_{\mvn}P_{\mvn}w_{\mvn}}
		=\frac{\sum_{\mvn}P_{\mvn}e^{\alpha\norm{m_{\mvn}}_2^2/2}f(m_{\mvn})}{\sum_{\mvn}P_{\mvn}e^{\alpha\norm{m_{\mvn}}_2^2/2}}+o(1).
	\end{align*}
	Choose $\alpha<\gamma<\alpha_c(\gk)$. Lemma~\ref{lem:expmoment} gives $\sup_N\E_{\mvgs}\big(e^{\gamma\norm m_2^2/2}\big)<\infty$, so $e^{\alpha\norm m_2^2/2}$ is uniformly integrable (trivially when $\alpha=0$). The central limit theorem for $m$ under the uniform measure therefore yields, with $g_0\sim\N(0,\Pi)$,
	\begin{align*}
		\big\langle f(m)\big\rangle_{\varrho_{N,\beta}}
		\to\frac{\E\big(f(g_0)e^{\alpha\norm{g_0}_2^2/2}\big)}{\E\big(e^{\alpha\norm{g_0}_2^2/2}\big)}
		=\E\big(f(g)\big)\quad\text{in probability}.
	\end{align*}
	The last equality holds because the quadratic tilt multiplies the covariance of $g_0$ on the range of $\Pi$ by $(1-\alpha)^{-1}$. This proves~\eqref{eq:general-color-clt}.

	For~\eqref{eq:general-color-tail}, the same total variation comparison gives
	\begin{align*}
		\big\langle\ind\{\norm m_2>L\}\big\rangle_{\varrho_{N,\beta}}
		=\frac{\sum_{\norm{m_{\mvn}}_2>L}P_{\mvn}w_{\mvn}}{\sum_{\mvn}P_{\mvn}w_{\mvn}}+o_{\pr}(1).
	\end{align*}
	The numerator is bounded by $Ce^{-cL^2}$ by~\eqref{eq:general-profile-tail}, and the denominator converges to a positive constant by~\eqref{eq:general-profile-tilt}. Finally, the coupling~\eqref{eq:truncation-coupling} transfers both assertions from $\wt J$ to $J$.
\end{proof}

%%%%%%%%%%%%%%%%%%%%%%%%%%%%%%%%%%%%%%%%%%%%%%%%%%%%%%%%%%%%%%%%%%%%%%%%%%%%%%%
\section{The joint quenched overlap limit}\label{sec:general-overlap}
%%%%%%%%%%%%%%%%%%%%%%%%%%%%%%%%%%%%%%%%%%%%%%%%%%%%%%%%%%%%%%%%%%%%%%%%%%%%%%%
We follow the comparison method of~\cite[Lemma 3.1]{CN95}, with the mixed log-moment kernels of Section~\ref{sec:general-kernels} in place of exact Gaussian identities.

For $T\geq0$, let $\cG_{(T)}$ be the set of continuous functions $F:\dR^{\gk\times\gk}\to\dR$ with $F(x)=o(e^{T\xi(x)})$ as $\norm{x}_F\to\infty$, normed by $\norm{F}_{(T)}:=\sup_x\abs{F(x)}e^{-T\xi(x)}$. For two profiles, let
\begin{align}
	U_N^{(\mvn),(\mvn')}(F)
	:=\E_{\mvgs,\mvgt}^{(\mvn),(\mvn')}
	\Big( e_N(\mvgs) e_N(\mvgt)
		e^{-\Lambda_N(\mvgs,\mvgt)}F\big(\sR(\mvgs,\mvgt)\big)\Big).
	\label{eq:general-UN}
\end{align}
We use the same notation when $F(\sR)$ is replaced by a deterministic function of the pair $(\mvgs,\mvgt)$.

\begin{lem}[Mixed-profile convergence lemma]\label{lem:general-CN-key}
	Fix $L>0$ and $0<\beta^2<T<\alpha_c(\gk)$, and let $K_0\sim\N(0,\Pi\otimes\Pi)$. For every $F\in\cG_{(T)}$,
	\begin{align}
		\sup_{\norm{m_{\mvn}}_2,\norm{m_{\mvn'}}_2\leq L}
		\E\Big(\Big| U_N^{(\mvn),(\mvn')}(F)
			- \cZ_N^{(\mvn)} \cZ_N^{(\mvn')}
				e^{-\varphi_{\gk,J}(\beta^2)}\E\big(F(K_0)\big)\Big|\Big)\to0.
		\label{eq:general-CN-key}
	\end{align}
\end{lem}

\begin{proof}
	Fix profiles in the stated window and abbreviate $U_N:=U_N^{(\mvn),(\mvn')}$.

	\smallskip\noindent\textbf{Step 1 (centered compactly supported functions).} First let $F$ be continuous and compactly supported, with $\E\big(F(K_0)\big)=0$. Write
	\begin{align*}
		U_N(F)=\E_{\mvgs}^{(\mvn)}\big(e_N(\mvgs)X_{\mvgs}\big),\quad
		X_{\mvgs}:=\E_{\mvgt}^{(\mvn')}\Big(e_N(\mvgt)e^{-\Lambda_N(\mvgs,\mvgt)}F\big(\sR(\mvgs,\mvgt)\big)\Big).
	\end{align*}
	Since $\E\big(\E_{\mvgs}^{(\mvn)}(e_N(\mvgs))\big)=1$, the density $e_N(\mvgs)$ defines a probability measure on the disorder and the uniform configuration $\mvgs$. The Cauchy--Schwarz inequality under this measure, followed by~\eqref{eq:general-triple-kernel}, gives
	\begin{align*}
		\Big(\E\big(\abs{U_N(F)}\big)\Big)^2
		&\leq\E\Big(\E_{\mvgs}^{(\mvn)}\big(e_N(\mvgs)X_{\mvgs}^2\big)\Big)\\
		&=\E_{\mvgs,\mvgt,\mvgt'}^{(\mvn),(\mvn'),(\mvn')}
		\Big(e^{\Lambda_N(\mvgt,\mvgt')+\fR_N^{(3)}(\mvgs,\mvgt,\mvgt')}
		F\big(\sR(\mvgs,\mvgt)\big)F\big(\sR(\mvgs,\mvgt')\big)\Big).
	\end{align*}
	Lemma~\ref{lem:general-kernels} shows that the remainders vanish in probability and are globally bounded. Lemma~\ref{lem:profile-clt} makes the three limiting overlaps independent, and Lemma~\ref{lem:profile-moments}(i) with exponent $T>\beta^2$ gives uniform integrability. Thus the last expression converges to
	\begin{align*}
		e^{\varphi_{\gk,J}(\beta^2)-\varphi_{\gk}(\beta^2)}\,
		\E\big(e^{\beta^2\xi(K_0)}\big)\,\E\big(F(K_0)\big)^2=0.
	\end{align*}

	\smallskip\noindent\textbf{Step 2 (boundedness).} The multiplier of $F(\sR)$ in~\eqref{eq:general-UN} has disorder expectation one, by the definition of $\Lambda_N$. Hence
	\begin{align*}
		\E\big(\abs{U_N(F)}\big)
		\leq\E_{\mvgs,\mvgt}^{(\mvn),(\mvn')}\big(\abs{F(\sR)}\big)
		\leq C(T,L)\norm{F}_{(T)},
	\end{align*}
	uniformly over the profile window, by Lemma~\ref{lem:profile-moments}(i) with exponent $T$.

	\smallskip\noindent\textbf{Step 3 (centered functions in $\cG_{(T)}$).} The centered compactly supported functions are dense in $\{F\in\cG_{(T)}:\E(F(K_0))=0\}$ for the norm $\norm{\cdot}_{(T)}$. Indeed, choose a continuous compactly supported $h\geq0$ with $\E\big(h(K_0)\big)=1$, and continuous cutoffs $0\leq\chi_r\leq1$ equal to one on the ball of radius $r$ and to zero outside the ball of radius $2r$. For centered $F$, define $F_r:=F\chi_r-\E\big(F(K_0)\chi_r(K_0)\big)h$. Then $\E\big(F_r(K_0)\big)=0$ and
	\begin{align*}
		\norm{F-F_r}_{(T)}
		\leq\sup_{\norm{x}_F\geq r}\abs{F(x)}e^{-T\xi(x)}
		+\abs{\E\big(F(K_0)\chi_r(K_0)\big)}\,\norm{h}_{(T)}\to0.
	\end{align*}
	The first term vanishes by the defining little-$o$ condition. The second vanishes by dominated convergence and $\E\big(F(K_0)\big)=0$, since $\abs{F(K_0)}\leq\norm{F}_{(T)}e^{T\xi(K_0)}$ is integrable by Lemma~\ref{lem:Zsecond}. Applying Step 1 to $F_r$ and Step 2 to $F-F_r$, and then letting $r\to\infty$, gives $U_N(F)\to0$ in $L^1$.

	\smallskip\noindent\textbf{Step 4 (general functions in $\cG_{(T)}$).} For general $F$, define $F_0:=F-e^{\beta^2\xi-\varphi_\gk(\beta^2)}\E\big(F(K_0)\big)$. Since $T>\beta^2$, we have $F_0\in\cG_{(T)}$, and $\E\big(F_0(K_0)\big)=0$ by Lemma~\ref{lem:Zsecond}. By Step 3 and linearity, it remains to prove~\eqref{eq:general-CN-key} for $F=e^{\beta^2\xi}$.
	The exact identity $U_N(e^{\Lambda_N})=\cZ_N^{(\mvn)}\cZ_N^{(\mvn')}$ and~\eqref{eq:general-pair-kernel} give
	\begin{align*}
		U_N(e^{\beta^2\xi})
		-e^{-\varphi_{\gk,J}(\beta^2)}\cZ_N^{(\mvn)}\cZ_N^{(\mvn')}\E\big(e^{\beta^2\xi(K_0)}\big)
		&=U_N\Big(e^{\beta^2\xi}-e^{\varphi_\gk(\beta^2)-\varphi_{\gk,J}(\beta^2)+\Lambda_N}\Big)\\
		&=U_N\Big(e^{\beta^2\xi}\big(1-e^{\fR_N^{(2)}}\big)\Big),
	\end{align*}
	where the pair arguments of $\Lambda_N$ and $\fR_N^{(2)}$ are suppressed. As in Step 2,
	\begin{align*}
		\E\Big(\Big|U_N\Big(e^{\beta^2\xi}\big(1-e^{\fR_N^{(2)}}\big)\Big)\Big|\Big)
		\leq\E_{\mvgs,\mvgt}^{(\mvn),(\mvn')}\Big(
		e^{\beta^2\xi(\sR)}\,\abs{1-e^{\fR_N^{(2)}(\mvgs,\mvgt)}}\Big)\to0.
	\end{align*}
	Here we used that $\fR_N^{(2)}$ is uniformly bounded and tends to zero in probability (Lemma~\ref{lem:general-kernels}), and that $e^{\beta^2\xi(\sR)}$ is uniformly integrable (Lemma~\ref{lem:profile-moments}(i) with exponent $T>\beta^2$). The convergence is uniform over the window, since all arguments apply along every sequence of profiles.
\end{proof}

\begin{proof}[Proof of Theorem~\ref{thm:C}]
	We first work with the truncated disorder. Let $f$ be a bounded continuous function of the overlap, and apply Lemma~\ref{lem:general-CN-key} to $F(x):=e^{\beta^2\xi(x)}f(x)$, which lies in $\cG_{(T)}$ for every $T\in(\beta^2,\alpha_c(\gk))$. Replacing $e^{\Lambda_N}$ by $e^{\varphi_{\gk,J}(\beta^2)-\varphi_\gk(\beta^2)+\beta^2\xi}$ as in Step 4 of that proof, and using~\eqref{eq:gaussian-tilt}, we obtain, uniformly over any fixed profile window,
	\begin{align}
		&\E_{\mvgs,\mvgt}^{(\mvn),(\mvn')}
			\Big( e_N(\mvgs) e_N(\mvgt)f\big(\sR(\mvgs,\mvgt)\big)\Big)
		- \cZ_N^{(\mvn)} \cZ_N^{(\mvn')}\,\E\big(f(K)\big)
			\to0\quad\text{in }L^1,
		\label{eq:general-overlap-factorization}
	\end{align}
	where $K\sim\N(0,\Pi\otimes\Pi/(1-\beta^2))$ as in Theorem~\ref{thm:C}. The fourth-cumulant factors have canceled.

	\emph{Balanced replicas.} Take both profiles equal to $\mvn^*$ with $\gk\mid N$. By~\eqref{eq:general-profile-lognormal}, we may divide~\eqref{eq:general-overlap-factorization} by $\big( \cZ_N^{\bal}\big)^2$ in probability. Since $c_N$ is constant on a profile class, the ratio is the overlap average under two independent balanced Gibbs replicas. This gives the balanced overlap limit throughout $\beta^2<\alpha_c(\gk)$.

	\emph{Unconstrained replicas.} Let $h_1,h_2$ be bounded continuous functions on $\dR^{\gk}$, and first restrict to $\norm{m_{\mvn}}_2,\norm{m_{\mvn'}}_2\leq L$. Multiply~\eqref{eq:general-overlap-factorization} by $P_{\mvn} P_{\mvn'}w_{\mvn}w_{\mvn'}h_1(m_{\mvn})h_2(m_{\mvn'})$ and sum over the two profiles. The weights $w_{\mvn}$ are bounded on this window, so the uniform $L^1$ convergence controls the summed error. Divide by $\big(\sum_{\mvn} P_{\mvn} w_{\mvn} \cZ_N^{(\mvn)}\big)^2$, whose inverse is bounded in probability by~\eqref{eq:general-gibbs-denominator}. Writing $\chi_L(\mvgs):=\ind\{\norm{m(\mvgs)}_2\leq L\}$, we obtain, for each fixed $L$, in probability,
	\begin{align*}
		&\Big\langle h_1(m(\mvgs^1))h_2(m(\mvgs^2))f\big(\sR(\mvgs^1,\mvgs^2)\big)\chi_L(\mvgs^1)\chi_L(\mvgs^2)\Big\rangle_{\varrho_{N,\beta}^{\otimes2}}\\
		&\quad-\Big\langle h_1(m)\chi_L\Big\rangle_{\varrho_{N,\beta}}
		\Big\langle h_2(m)\chi_L\Big\rangle_{\varrho_{N,\beta}}\,\E\big(f(K)\big)
		\longrightarrow0.
	\end{align*}
	Removing the cutoffs changes both terms by at most a constant times $\langle\ind\{\norm m_2>L\}\rangle_{\varrho_{N,\beta}}$. Letting $L\to\infty$ and using the tail bound~\eqref{eq:general-color-tail}, we obtain, in probability,
	\begin{align*}
		\Big\langle h_1(m(\mvgs^1))h_2(m(\mvgs^2))f\big(\sR(\mvgs^1,\mvgs^2)\big)\Big\rangle_{\varrho_{N,\beta}^{\otimes2}}
		-\Big\langle h_1(m)\Big\rangle_{\varrho_{N,\beta}}
		\Big\langle h_2(m)\Big\rangle_{\varrho_{N,\beta}}\,\E\big(f(K)\big)
		\longrightarrow0.
	\end{align*}
	The color limits~\eqref{eq:general-color-clt} now yield
	\begin{align*}
		\Big\langle h_1(m(\mvgs^1))h_2(m(\mvgs^2))f\big(\sR(\mvgs^1,\mvgs^2)\big)\Big\rangle_{\varrho_{N,\beta}^{\otimes2}}
		\to\E\big(h_1(g)\big)\E\big(h_2(g')\big)\E\big(f(K)\big)\quad\text{in probability},
	\end{align*}
	with $g,g',K$ independent as in Theorem~\ref{thm:C}.

	It remains to pass from products to general $F$. The marginal limits give tightness in probability of the Gibbs law of $(m(\mvgs^1),m(\mvgs^2),\sR(\mvgs^1,\mvgs^2))$. Finite sums of products $h_1(x)h_2(y)f(z)$ form an algebra that separates points and contains the constants. By the Stone--Weierstrass theorem they approximate every continuous function uniformly on compact sets. Together with tightness, the factorization therefore proves the joint weak convergence in probability. The identity $m=\sqrt{\gk N}(p-u)$ also gives the corresponding limit for the color fractions. Finally, the coupling~\eqref{eq:truncation-coupling} transfers all assertions to the original disorder.
\end{proof}

\begin{rem}[Fluctuations of the pair empirical measure]\label{cor:pair-clt-general}
	Under the hypotheses of Theorem~\ref{thm:C},
	\begin{align*}
		\sqrt N\big(p(\mvgs^1,\mvgs^2)-\bar u\big)
		\Longrightarrow \frac1\gk K
			+\frac1{\sqrt{\gk}}
				\big(g u^{\intercal}+u(g')^{\intercal}\big)
	\end{align*}
	weakly in probability, with $g,g',K$ as in Theorem~\ref{thm:C}. The covariance between the entries $(a,b)$ and $(c,d)$ of the limit is
	\begin{align*}
		\frac{\Pi_{ac}\Pi_{bd}}{\gk^2(1-\beta^2)}
			+\frac{\Pi_{ac}+\Pi_{bd}}{\gk^3(1-\alpha)}.
	\end{align*}
	Indeed, the margins of $p(\mvgs^1,\mvgs^2)$ are $p(\mvgs^1)$ and $p(\mvgs^2)$, so expanding $\Pi p(\mvgs^1,\mvgs^2)\Pi$ gives the exact decomposition
	\begin{align*}
		p(\mvgs^1,\mvgs^2)-\bar u
		=\Pi p(\mvgs^1,\mvgs^2)\Pi
		+\big(p(\mvgs^1)-u\big)u^{\intercal}
		+u\big(p(\mvgs^2)-u\big)^{\intercal}.
	\end{align*}
	By~\eqref{eq:pair-profile} and~\eqref{eq:profile-mag}, multiplying by $\sqrt N$ turns the right-hand side into
	$\gk^{-1}\sR(\mvgs^1,\mvgs^2)
		+\gk^{-1/2}\big(m(\mvgs^1)u^{\intercal}+u\,m(\mvgs^2)^{\intercal}\big)$.
	The claim follows from Theorem~\ref{thm:C}. The first term contributes $\Pi_{ac}\Pi_{bd}/(\gk^2(1-\beta^2))$ to the covariance, and the two rank-one terms contribute $\Pi_{ac}/(\gk^3(1-\alpha))$ and $\Pi_{bd}/(\gk^3(1-\alpha))$.
\end{rem}

\appendix

%%%%%%%%%%%%%%%%%%%%%%%%%%%%%%%%%%%%%%%%%%%%%%%%%%%%%%%%%%%%%%%%%%%%%%%%%%%%%%%
\section{Entropy and exponential moments}\label{sec:constants}
%%%%%%%%%%%%%%%%%%%%%%%%%%%%%%%%%%%%%%%%%%%%%%%%%%%%%%%%%%%%%%%%%%%%%%%%%%%%%%%
Both temperature conditions compare a quadratic energy with an entropy. The relevant constant is the optimal sub-Gaussian constant of a bounded field, and this appendix computes it for the Potts field $\phi$.
Lemmas~\ref{lem:expmoment} and~\ref{lem:eta-dual} concern a general bounded mean-zero field on a finite set. The first gives an exponential moment bound for its normalized sums, and the second a dual formula for the optimal constant $\eta_W$ in terms of relative entropy. Lemma~\ref{lem:eta-phi} evaluates the constant $\eta_\phi(\gk)$ of the Potts field. We conclude with the proof of Lemma~\ref{lem:near-uniform}, the entropy bound for pair measures with near-uniform margins used in Section~\ref{sec:profile}.

\subsection{Exponential moments and entropy duality}

\begin{lem}\label{lem:expmoment}
	Let $W, W_1, W_2, \ldots$ be i.i.d.~mean-zero random vectors in $\dR^d$ with $\norm{W}_2 \leq b$ almost surely. Set $S_N := N^{-1/2} \sum_{i=1}^{N} W_i$, and define the \emph{optimal sub-Gaussian constant} of $W$ by
	\begin{align}
		\eta_W := \sup_{s \in \dR^d \setminus \{0\}} \frac{2 \ln \E\big(e^{\langle s, W \rangle}\big)}{\norm{s}_2^2}.
		\label{eq:eta-def}
	\end{align}
	Then $\eta_W \leq b^2$. If $\eta_W < 1$, then for every $N \geq 1$,
	\begin{align*}
		\E\Big( \exp\Big( \frac{1}{2} \norm{S_N}_2^2 \Big)\Big) \leq \left( 1 - \eta_W \right)^{-d/2}.
	\end{align*}
\end{lem}
\begin{proof}
	Write $F(s) := \ln \E\big(e^{\langle s, W \rangle}\big)$. By Hoeffding's lemma~\cite{Hoe63}, every mean-zero random variable $X$ with $\abs{X} \leq M$ almost surely satisfies $\E(e^{X}) \leq e^{M^2/2}$. The variable $X = \langle s, W \rangle$ is mean-zero and satisfies $\abs{X} \leq \norm{s}_2\, b$ by the Cauchy--Schwarz inequality. Hence $F(s) \leq \norm{s}_2^2 b^2/2$ and $\eta_W \leq b^2 < \infty$. By the definition of $\eta_W$,
	\begin{align}
		F(s) \leq \frac{\eta_W}{2} \norm{s}_2^2 \quad \text{for every } s \in \dR^d.
		\label{eq:subgaussian}
	\end{align}
	Now let $\zeta$ be a standard Gaussian vector in $\dR^d$, independent of $(W_i)_{i \geq 1}$. By the Hubbard--Stratonovich identity~\cite{Hub59}, $e^{\norm{x}_2^2/2} = \E_\zeta\big(e^{\langle \zeta, x \rangle}\big)$, and by~\eqref{eq:subgaussian} with $s = \zeta/\sqrt{N}$,
	\begin{align*}
		\E\Big( \exp\Big( \frac{1}{2} \norm{S_N}_2^2 \Big)\Big)
		= \E_\zeta\Big(  \E_W \big( e^{\langle \zeta, S_N \rangle} \big)\Big)
		= \E_\zeta\big(  e^{N F(\zeta/\sqrt N)}\big)
		\leq \E_\zeta\big(  e^{\frac{\eta_W}{2} \norm{\zeta}_2^2}\big)
		= \left( 1 - \eta_W \right)^{-d/2}.
	\end{align*}
	The last equality uses $\eta_W < 1$.
\end{proof}

Relative entropy is the rate function of the empirical measure, and the natural quantity against which $\eta_W$ is measured.

\begin{lem}[Dual description of $\eta_W$]\label{lem:eta-dual}
	Let $W$ be a mean-zero random vector on a finite probability space $(\cX, \pi)$, and let $\cP(\cX)$ be the set of probability measures on $\cX$. Then
	\begin{align*}
		\eta_W = \sup_{q \in\cP(\cX),\, q\neq \pi} \frac{\norm{\E_q(W)}_2^2}{2  D(q  \|  \pi)}.
	\end{align*}
\end{lem}

\begin{proof}
	The Gibbs variational principle on the finite space $(\cX,\pi)$ states that, for every function $f$,
	\begin{align*}
		\ln \E_\pi\big(e^{f}\big) = \sup_{q\in\cP(\cX)} \big[ \E_q(f) - D(q\|\pi) \big],
	\end{align*}
	see~\cite{DZ98}. Apply it to $f = \langle s, W\rangle$. By the definition~\eqref{eq:eta-def}, for every $c>0$,
	\begin{align*}
		\eta_W \leq c
		\iff \forall q, s : \langle s, \E_q(W) \rangle - \frac{c}{2}\norm{s}_2^2 \leq D(q\|\pi)
		\iff \forall q : \frac{\norm{\E_q(W)}_2^2}{2c} \leq D(q\|\pi).
	\end{align*}
	The last step maximizes the quadratic function of $s$. Since this equivalence holds for every $c > 0$, the claim follows.
\end{proof}

\begin{rem}\label{rem:expmoment-opt}
	(a) \emph{The condition $\eta_W<1$ in Lemma~\ref{lem:expmoment} is sharp.} Write $F(s):=\ln\E\big(e^{\langle s,W\rangle}\big)$ and let $\zeta\sim\N(0,I_d)$. The proof of Lemma~\ref{lem:expmoment} gives the exact identity
	\begin{align*}
		\E \Big(\exp\big(\tfrac12 \norm{S_N}_2^2\big)\Big)=\E_\zeta\big(e^{N F(\zeta/\sqrt N)}\big).
	\end{align*}
	Suppose that $2F(s_0) > \norm{s_0}_2^2$ for some $s_0$. Restricting the Gaussian integral to the unit ball around $\sqrt N s_0$ and using the local Lipschitz continuity of $F$ gives
	\begin{align*}
		\E \Big(\exp \big( \tfrac12 \norm{S_N}_2^2 \big)\Big) \geq c_d\,  e^{N \left( F(s_0) - \norm{s_0}_2^2/2 \right) - \cO(\sqrt N)} \to \infty.
	\end{align*}
	(b) For a Rademacher variable ($d = 1$, $W = \pm 1$), $F(s) = \ln\cosh(s) \leq s^2/2$ and $\eta_W = 1$. In this case~\eqref{eq:subgaussian} is the classical inequality behind the second-moment computation of~\cite{ALR87} for the SK model. Lemma~\ref{lem:eta-phi} below computes the corresponding constant for the $\gk$-state field $\phi$.
\end{rem}

\subsection{Potts constants and near-uniform margins}

We now turn to the field $\phi$. Recall that $\phi(x)=\sqrt\gk(e_x-u)$ and let $\sigma$ be uniform on $[\gk]$. Its covariance is $\E_{\sigma}\big(\phi(\sigma)\phi(\sigma)^{\intercal}\big)=\Pi$, the orthogonal projection onto $\cV$ defined in~\eqref{def:phi}.

The next lemma computes $\eta_\phi(\gk)$. The extremal direction is ``one color against the rest''. Its tilted law $p^*$ is the ordered state of the Curie--Weiss--Potts model at its critical temperature~\cite{EW90}, which explains the appearance of the Curie--Weiss--Potts critical value.
The equivalent entropy inequality for $\gk\geq3$ follows from~\cite[Proposition 10 and Section 5]{AN05}. We include a proof in the present sub-Gaussian formulation, which also identifies a maximizing tilt.

\begin{lem}[The single-field constant]\label{lem:eta-phi}
	Let $\Lambda(t) := \ln \E_\sigma\big(  e^{\langle t, \phi(\sigma) \rangle}\big)$ for $t \in \dR^\gk$. Then
	\begin{align*}
		\sup_{t \neq 0} \frac{2 \Lambda(t)}{\norm{t}_2^2} = \eta_\phi(\gk) =
		\begin{cases}
			1,                                                  & \gk = 2,    \\[2pt]
			\dfrac{\gk(\gk-2)}{2(\gk-1)\ln(\gk-1)}, & \gk \geq 3.
		\end{cases}
	\end{align*}
	For $\gk \geq 3$ we have $\eta_\phi(\gk)>1$, and the supremum is attained at the $\gk$ points $t = \tfrac{2\ln(\gk-1)}{\gk}\,  \phi(x_0)$, $x_0 \in [\gk]$. The corresponding tilted law is, up to a permutation of colors, the ``one color against the rest'' profile
	\begin{align*}
		p^* = \Big( \frac{\gk-1}{\gk}, \frac{1}{\gk(\gk-1)}, \ldots, \frac{1}{\gk(\gk-1)} \Big).
	\end{align*}
	Equivalently, by Lemma~\ref{lem:eta-dual},
	\begin{align*}
		\frac{1}{\eta_\phi(\gk)} = \inf_{p \neq u} \frac{2 D(p  \|  u)}{\gk \norm{p - u}_2^2},
	\end{align*}
	where the infimum runs over probability vectors $p$ on $[\gk]$.
\end{lem}

\begin{proof}
	\emph{Case $\gk = 2$.} Here $\phi(\sigma) = \varepsilon (1,-1)/\sqrt 2$ with $\varepsilon$ a Rademacher variable. Thus $\Lambda(t) = \ln\cosh\big( (t_1 - t_2)/\sqrt 2 \big) \leq (t_1-t_2)^2/4 \leq \norm{t}_2^2/2$, and the ratio tends to $1$ along $t = (a, -a)$ as $a \to 0$. Hence the supremum equals $1$. Assume $\gk \geq 3$ from now on.

	\smallskip\noindent\textbf{Step 1 (reduction to centered vectors).} We have $\langle t, \phi(\sigma) \rangle = \sqrt\gk\,  (t_\sigma - \bar t)$ with $\bar t := \gk^{-1}\sum_x t_x$. So $\Lambda$ depends only on the centered part $\hat t := t - \bar t\,  \vone $, while $\norm{\hat t}_2 \leq \norm{t}_2$. We may therefore restrict to centered $t$. Writing $c := \sqrt\gk\,  \hat t \in \cV$, the claim becomes
	\begin{align*}
		\sup_{c \in \cV \setminus \{0\}} Q(c) = \eta_\phi(\gk), \quad Q(c) := \frac{2\gk  L(c)}{\norm{c}_2^2}, \quad L(c) := \ln\Big( \frac1\gk \sum_{x} e^{c_x} \Big).
	\end{align*}

	\smallskip\noindent\textbf{Step 2 (a binary sub-problem).} For $\rho \in (0,1)$ let $\Lambda_\rho(s) := \ln\big( \rho  e^{(1-\rho)s} + (1-\rho)  e^{-\rho s} \big)$ be the cumulant generating function of a centered Bernoulli$(\rho)$ variable, and define
	\begin{align*}
		h(\rho) := \frac{2\rho(1-\rho)\ln\frac{1-\rho}{\rho}}{1 - 2\rho} \quad (\rho \neq \tfrac12), \quad h(\tfrac12) := 1.
	\end{align*}
	The optimal sub-Gaussian constant of such a variable is known exactly:
	\begin{align}
		\sup_{s \neq 0} \frac{2 \Lambda_\rho(s)}{\rho(1-\rho) s^2} = \frac{1}{h(\rho)}, \quad \text{attained at } s^* = 2\ln\tfrac{1-\rho}{\rho} \text{ for } \rho \neq \tfrac12.
		\label{eq:binary-sup}
	\end{align}
	The inequality $\Lambda_\rho(s) \leq \tfrac{\rho(1-\rho)}{2h(\rho)}s^2$, with the constant $\rho(1-\rho)/h(\rho) = (1-2\rho)/\big(2\ln\tfrac{1-\rho}\rho\big)$, is the Kearns--Saul inequality~\cite{KS98}. See~\cite{BK13,BM13} for complete proofs and the optimality of the constant. At $\rho = \tfrac12$ it is the classical bound $\ln\cosh(s/2) \leq s^2/8$. That the supremum is attained at $s^*$ follows by substitution. Indeed, $e^{s^*} = \big(\tfrac{1-\rho}\rho\big)^2$ gives $\rho e^{s^*} + 1 - \rho = \tfrac{1-\rho}\rho$, so that
	\begin{align*}
		\Lambda_\rho(s^*) = -\rho s^* + \ln\tfrac{1-\rho}{\rho} = \tfrac{1-2\rho}{2}\,  s^* = \tfrac{\rho(1-\rho)}{2h(\rho)} (s^*)^2,
	\end{align*}
	where the last step uses the definition of $h$. For $\rho = \tfrac12$ the supremum is only approached as $s \to 0$.

	\smallskip\noindent\textbf{Step 3 (two-valued reduction).} We first record the lower bound. Let $c^* := 2\ln(\gk-1) \big( e_1 - \tfrac1\gk \vone  \big)$. It takes one value on a single coordinate and another value on the remaining $\gk - 1$ coordinates. The computation of Step 4 below with $j = 1$ and the equality case of~\eqref{eq:binary-sup} give $Q(c^*) = 1/h(1/\gk) = \eta_\phi(\gk)$. Moreover $\eta_\phi(\gk) > 1$ for $\gk \geq 3$. With $x := \gk - 1 > 1$ this reads $x^2 - 1 > 2x\ln x$, that is, $x - 1/x - 2\ln x > 0$. This holds since the left side vanishes at $x = 1$ and has derivative $(1 - 1/x)^2 > 0$.

	As $c \to 0$ in $\cV$, we have $L(c) = \norm{c}_2^2/(2\gk) + \cO(\norm{c}_2^3)$ because $\sum_x c_x = 0$, so $Q(c) \to 1$. Also $L(c) \leq \max_x c_x \leq \norm{c}_2$, so $Q(c) \to 0$ as $\norm{c}_2 \to \infty$. Since $\sup Q \geq \eta_\phi(\gk) > 1$, every maximizing sequence stays in a compact subset of $\cV \setminus \{0\}$. Hence the supremum is attained at some $c^0 \neq 0$, which is a critical point of $Q$ on $\cV$. Let $\nabla_{\cV}=\Pi \nabla$ denote the gradient along $\cV$. Since $\partial_x L(c) = e^{c_x}/\sum_y e^{c_y}$ and these $\gk$ numbers sum to one,
	\begin{align*}
		\big( \nabla_{\cV} L(c) \big)_x = \frac{e^{c_x}}{\sum_y e^{c_y}} - \frac1\gk.
	\end{align*}
	The criticality condition for $Q$ reads $\nabla_{\cV} L(c^0) = 2\big( L(c^0)/\norm{c^0}_2^2 \big)\,  c^0$. Multiplying its $x$-th component by $\sum_y e^{c^0_y}$ gives
	\begin{align*}
		e^{c^0_x} = A  c^0_x + B \quad \text{for every } x \in [\gk],
	\end{align*}
	where $A$ and $B$ depend on $c^0$ but not on $x$.
	By the strict convexity of the exponential, this equation has at most two solutions. Hence $c^0$ takes exactly two distinct values, since a single value would force $c^0 = 0$.

	\smallskip\noindent\textbf{Step 4 (lumping and optimization over $j$).} Let $c \in \cV$ take the value $a$ on $j$ coordinates and $b$ on the remaining $\gk - j$ coordinates, with $1 \leq j \leq \gk - 1$ and $j a + (\gk - j) b = 0$. Set $\rho := j/\gk$ and $s := a - b$, so that $a = (1-\rho)s$ and $b = -\rho s$. Then
	\begin{align*}
		L(c) = \ln\big( \rho e^{a} + (1-\rho) e^{b} \big) = \Lambda_\rho(s), \quad \norm{c}_2^2 = \gk\big( \rho a^2 + (1-\rho) b^2 \big) = \gk  \rho(1-\rho)  s^2,
	\end{align*}
	whence $Q(c) = 2\Lambda_\rho(s) / (\rho(1-\rho) s^2)$. By~\eqref{eq:binary-sup}, $\sup Q = \max_{1 \leq j \leq \gk-1} 1/h(j/\gk)$. Finally, $h$ is symmetric about $\tfrac12$ and strictly increasing on $(0, \tfrac12]$. Indeed, the substitution $w := 1 - 2\rho \in [0, 1)$ gives $h = (1-w^2)  \tanh^{-1}(w)/w$, and
	\begin{align*}
		\frac{d h}{d w} = \frac1w - \frac{(1+w^2)\tanh^{-1}(w)}{w^2} < 0
		\iff (1+w^2)\tanh^{-1}(w) > w,
	\end{align*}
	which holds for $w \in (0,1)$ since $\tanh^{-1}(w) > w$. Hence the maximum over $j$ is attained at $j = 1$ (equivalently $j = \gk - 1$), and
	\begin{align*}
		\sup_{c} Q(c) = \frac{1}{h(1/\gk)} = \frac{\gk(\gk-2)}{2(\gk-1)\ln(\gk-1)} = \eta_\phi(\gk).
	\end{align*}
	The maximizer is the two-valued vector with $j = 1$ and gap $s^* = 2\ln(\gk-1)$, namely $c^* = 2\ln(\gk-1)(e_{x_0} - \tfrac1\gk\vone )$. It corresponds to $t = c^*/\sqrt\gk = \tfrac{2\ln(\gk-1)}{\gk}\,  \phi(x_0)$. Its tilted law is proportional to $(e^{c^*_x})_x$, which is exactly $p^*$. Finally, the dual formula is Lemma~\ref{lem:eta-dual} applied to $\phi$ on $([\gk], u)$, since $\E_p(\phi) = \sqrt\gk\,  (p - u)$ and hence $\norm{\E_p(\phi)}_2^2 = \gk \norm{p-u}_2^2$.
\end{proof}

The dual formula in Lemma~\ref{lem:eta-phi} gives, for every probability vector $p$ on $[\gk]$,
\begin{align}
	2D(p\|u)\geq\gk\alpha_c(\gk)\norm{p-u}_2^2.
	\label{eq:row-dual}
\end{align}
Applying this inequality to the conditional rows of a pair measure gives the entropy bound of Lemma~\ref{lem:near-uniform}.

\begin{proof}[Proof of Lemma~\ref{lem:near-uniform}]
	The idea is to apply~\eqref{eq:row-dual} to each conditional row and to average, using the chain rule for relative entropy. The error terms come from replacing $u$ by $p'$ and $1/p_x$ by $\gk$. Write $p^{(x)} := q_{x\cdot}/p_x \in \cP([\gk])$ for the $x$-th conditional row of $q$. Since $\gk u_y = 1$,
	\begin{align*}
		D(p^{(x)}\|p') - D(p^{(x)}\|u) = - \sum_y p^{(x)}_y \ln(\gk p'_y) = - \gk \big\langle p^{(x)}, p'-u \big\rangle - \sum_y p^{(x)}_y \Big( \ln(\gk p'_y) - \gk (p'_y - u_y) \Big).
	\end{align*}
	For the first term, the Cauchy--Schwarz inequality and $\langle u, p'-u\rangle = 0$ give
	\begin{align*}
		\abs{\big\langle p^{(x)}, p'-u \big\rangle}
		 & = \abs{\big\langle p^{(x)} - p',  p'-u \big\rangle + \big\langle p',  p'-u \big\rangle} \\
		 & \leq \norm{p^{(x)} - p'}_2 \norm{p' - u}_2 + \norm{p' - u}_2^2
		\leq C\delta  \norm{p^{(x)} - p'}_2 + C\delta^2.
	\end{align*}
	For the second term, note that $\ln(\gk p'_y) = \ln\big(1 + \gk(p'_y - u_y)\big)$ and $\abs{\gk(p'_y-u_y)} \leq \gk\delta \leq \tfrac12$. Hence $\abs{\ln(\gk p'_y) - \gk(p'_y - u_y)} \leq C\delta^2$, and $\sum_y p^{(x)}_y = 1$.
	We combine these bounds with~\eqref{eq:row-dual} and with $\norm{p^{(x)}-u}_2^2 \geq \norm{p^{(x)}-p'}_2^2 - 2\sqrt\gk\, \delta\norm{p^{(x)}-p'}_2$, which follows from $\norm{p'-u}_2\leq\sqrt\gk\,\delta$. This gives
	\begin{align*}
		D(p^{(x)}\|p')  \geq \frac{\alpha_c\gk}{2}  \norm{p^{(x)}-p'}_2^2 - C\delta  \norm{p^{(x)} - p'}_2 - C\delta^2.
	\end{align*}

	Now average over rows. By the chain rule, $D(q\|p\otimes p') = \sum_x p_x  D(p^{(x)}\|p')$. Since $p^{(x)}_y - p'_y = (q_{xy} - p_x p'_y)/p_x$,
	\begin{align*}
		\sum_x p_x \norm{p^{(x)} - p'}_2^2 = \sum_{x,y} \frac{w_{xy}^2}{p_x}.
	\end{align*}
	The margin bound $\gk^{-1}(1 - \gk\delta) \leq p_x \leq \gk^{-1}(1 + \gk\delta)$ bounds this quantity below by $\gk(1-\gk\delta)\norm w_F^2$ and above by $C\norm w_F^2$.
	The lower bound produces the error $\alpha_c\gk^3\delta\norm w_F^2/2$, which is at most $C\delta\norm w_F$ because $\norm w_F\leq\norm q_F+\norm{p\otimes p'}_F\leq2$. The upper bound and the Cauchy--Schwarz inequality in $x$ bound the averaged linear term by $C\delta\norm w_F$. Together these estimates give~\eqref{eq:near-uniform}.
\end{proof}

\section*{Use of large language models}
The authors proved Theorem~\ref{thm:full_clt} on a smaller range of $\beta$ using the generalized $\mathrm{Ch}_\gk,\mathrm{Th}_\gk$ expansion (see Remark~\ref{rem:loop-motivation}) and planned to treat the Gaussian case for $\beta<\beta_{\max}(\gk)$ using the stochastic-calculus method of~\cite{CN95}. They provided Claude with this plan and a sketch of the proof, to generate a draft. This draft was subsequently used in ChatGPT to simplify the proof for general symmetric disorder with unit variance and finite fourth moment. The authors have written and verified every statement and proof in this article, and take full responsibility for its correctness.

\bibliographystyle{abbrv}
\bibliography{Potts}
\end{document}